\documentclass[11pt]{article}
\usepackage{graphicx}               % Necessary to use \scalebox

\usepackage{amsmath, amssymb, amsthm,bm,bbm} %DIF >
\usepackage[T1]{fontenc}
\usepackage{dsfont}
\usepackage[colorlinks=true, allcolors=blue,backref]{hyperref}
\usepackage[capitalize,nameinlink]{cleveref}
\usepackage{url}
\usepackage[utf8]{inputenc}

\usepackage{tikz}
\usepackage{comment}

\allowdisplaybreaks
\expandafter\let\expandafter\oldproof\csname\string\proof\endcsname
\let\oldendproof\endproof
\renewenvironment{proof}[1][\proofname]{%
	\oldproof[\bf #1]%
}{\oldendproof}

\allowdisplaybreaks

\newcommand*\samethanks[1][\value{footnote}]{\footnotemark[#1]}

\theoremstyle{plain}
\newtheorem{lemma}{Lemma}[section]
\newtheorem{theorem}[lemma]{Theorem}
\newtheorem{claim}[lemma]{Claim}
\newtheorem{proposition}[lemma]{Proposition}

\newtheorem{definition}[lemma]{Definition}

\newcommand{\mb}{\mathbb}

\newcommand{\mc}{\mathcal}

\DeclareMathOperator{\spann}{span}
\RequirePackage[normalem]{ulem} %DIF PREAMBLE
\RequirePackage{color}\definecolor{RED}{rgb}{1,0,0}\definecolor{BLUE}{rgb}{0,0,1} %DIF PREAMBLE
\newcommand{\abs}[1]{\left\vert {#1} \right\vert}

\newcommand{\floor}[1]{\left\lfloor {#1} \right\rfloor}
\newcommand{\ceil}[1]{\left\lceil {#1} \right\rceil}

\newcommand{\dist}{\text{dist}}

\DeclareMathOperator{\pd}{pd}
\DeclareMathOperator{\tpd}{tpd}
\DeclareMathOperator{\conv}{conv}
\newcommand{\EE}{{\mathbb{E}}}

\newcommand{\free}{{\varphi}}
\newcommand{\fr}{{\widetilde{\free}}}
\newcommand{\FF}{\mathcal F}
\newcommand{\KK}{\mathcal K}
\newcommand{\restrict}[2]{{{#1}\big\vert_{\scriptscriptstyle {#2}}}}

\global\long\def\zj#1{\textcolor{cyan}{\textbf{[ZJ comments:} #1\textbf{]}}}
\title{ \vspace{-0.9cm} The Helly number of Hamming balls and related problems }

\author{
Noga Alon
\thanks{Princeton University,
Princeton, NJ, USA
and
Tel Aviv University, Tel Aviv,
Israel.
Email: {\tt nalon@math.princeton.edu}.
Research supported in part by
NSF grant DMS-2154082.}
\and Zhihan Jin
\thanks{
ETH, Z\"urich, Switzerland.
Email: {\tt \{zhihan.jin,benjamin.sudakov\}@math.ethz.ch}.
Research supported in part by  SNSF grant   200021-228014.
} 
\and 
Benny Sudakov\samethanks
}

\date{}

\begin{document}

\maketitle

\begin{abstract}
    We prove the following variant of Helly's classical theorem 
	for Hamming balls with a bounded radius.
    For $n>t$ and any (finite or infinite) set $X$, if in a family 
	of Hamming balls of radius $t$ in $X^n$, every subfamily of at 
	most $2^{t+1}$ balls have a common point, so do all members of the family.
    This is tight for all $|X|>1$ and all $n>t$. The proof of the 
	main result is based on a novel variant of the so-called 
	dimension argument, which allows one to prove upper bounds 
	that do not depend on the dimension of the ambient space.
    We also discuss several related questions and connections to problems and results in extremal finite set theory and graph theory. 
\end{abstract}

\section{Introduction}

\subsection{Helly-type problems for the Hamming balls}

Helly's theorem, proved by Helly more than 100 years ago (\cite{helly23}),
is a fundamental result in Discrete Geometry. It asserts that a finite
family
of convex sets in the $d$-dimensional Euclidean 
space has a nonempty intersection if  every subfamily of
at most $d+1$ of the sets has a nonempty intersection. 

This theorem, in which the number $d+1$ is tight, led to numerous fascinating variants and extensions in geometry and beyond (c.f., e.g., \cite{Ec93,BK22} for two survey articles). 
It motivated the definition of the {\em Helly number} $h(\FF)$ for a general family $\FF$ of sets. This 
is the smallest integer $h$ such that for any finite subfamily 
$\KK$ of $\FF$, if every subset of at most $h$ members of $\KK$ has a nonempty intersection then all sets in $\KK$ have a nonempty intersection.
The classical theorem of Helly asserts that the Helly number of the family of
convex sets in $\mb{R}^d$ is $d+1$. An additional example of a 
known Helly number is the Theorem of Doignon \cite{Do73} that asserts that the Helly number of convex lattice sets in $d$-space, that is sets of the form $C \cap Z^d$ where $C$ is a convex set in $\mb{R}^d$, is $2^d$. 
A more combinatorial example is the fact that the Helly number of the collection of (sets of vertices of) subtrees of any tree is $2$.

In the spaces $X^n$ for finite or infinite $X$, the {\em Hamming balls} 
are among the most natural objects to study.
The {\em Hamming distance} between $p,q\in X^n$, denoted by $\dist(p,q)$, is the number of coordinates where $p$ and $q$ differ,
and the {Hamming ball} of radius $t$ centered at $x \in X^n$, denoted by 
$B(x,t)$, is the set of all points $p \in X^n$ that satisfy 
$\dist(p,x) \le t$.
Note that every Hamming ball of radius $t$ is the whole space if $n \le t$.
Hence, we may and will always assume that $n \geq t+1$.
Our main result in the present paper is the determination of the Helly number of the family of all Hamming balls of radius $t$ in the space $X^n$, where $X$ is an arbitrary (finite or infinite) set. 
\begin{theorem}
\label{t11}
    Let $n > t \ge 0$ and $X$ be any set of cardinality $|X|\ge 2$.
    The Helly number $h(n,t;X)$ of the family of all Hamming balls of radius $t$ in $X^n$ is exactly $2^{t+1}$.
\end{theorem}
Crucially, $h(n,t;X)$ depends only on $t$.
We note that the special case $X=\{0,1\}$ of this theorem settles a 
recent problem raised in \cite{RST23}, where the question is motivated by an application in learning theory. 
See also \cite{BHMZ20} for more about the connection between Helly numbers and questions in computational learning.
% The proof further shows that the Helly number of the family of all Hamming balls of radius at most $t$ in $X^n$ is also $2^{t+1}$.

Another fundamental result in Discrete Geometry is Radon’s
theorem~\cite{Rad21} which states that any set of $d+2$ points in the $d$-dimensional Euclidean space can be partitioned into two parts whose convex hulls intersect.
This was first obtained by Radon in 1921 and was used to prove Helly's theorem; see also~\cite{Ec93,BK22}. Using our methods we can prove the following strengthening of Theorem \ref{t11}. As we explain below it can be viewed as Radon's theorem for the Hamming balls.
\begin{theorem}\label{thm: general helly}
    Let $n>t \ge 0$ and $X$ be any set of cardinality $|X|\ge 2$.
    If $B_1,B_2,\dots,B_m$ are Hamming balls in $X^n$ of radius $t$, then there exists $I\subseteq [m]$ of size at most $2^{t+1}$ such that $\bigcap_{i=1}^m B_i=\bigcap_{i\in I} B_i$.
\end{theorem}
\noindent
It is easy to see, that the upper bound of the Helly number $h(n,t;X)\le 2^{t+1}$ follows from this result by taking $\bigcap_{i=1}^m B_i=\emptyset$.
To explain the connection of this statement with Radon's theorem we 
briefly discuss the notion of abstract {\em convexity spaces}.

An (abstract) convexity space is a pair $(U,\mc{C})$ where $U$ is a nonempty set and $\mc{C}$ is a family of subsets of $U$ satisfying the following properties. Both $\emptyset$ and $U$ are in $\mc{C}$ and the intersection of any subfamily of sets in $\mc{C}$ is a set in $\mc{C}$. 
One natural example is the standard Euclidean convexity space $(\mb{R}^d,\mc{C}^d)$ where $\mc{C}^d$ is the family of all convex sets in $\mb{R}^d$.
We refer the readers to the book by van de Vel~\cite{Vel93} for a comprehensive overview of the theory of convexity spaces.

In a convexity space $(U, \mc{C})$, the members of $\mc{C}$ are called {\em convex sets}. 
Given a subset $Y \subseteq U$, the {\em convex hull} of $Y$, denoted by $\conv(Y)$, is the intersection of all the convex sets containing $Y$, that is the minimal convex set containing $Y$.
The {\em Radon number} of $(U,\mc{C})$, denoted by $r(\mc{C})$, is the smallest integer $r$ (if it exists) such that any subset $P\subseteq X$ of at least $r$ points can be partitioned into two parts $P_1$ and $P_2$ such
that $\conv(P_1)\cap \conv(P_2)\not = \emptyset$.
For instance, $r(\mc{C}^d)=d+2$ for the family of convex sets in $\mc{R}^d$.
It is well-known that the Helly number is smaller than the Radon number if the latter is finite; see~\cite{Lev51}.

In our case, $U=X^n$, and $\mc{C}_H$ contains all the intersections of Hamming balls of radius $t$. For simplicity we assume that $X$ is finite. Then, $(U,\mc{C}_H)$ is automatically a convexity space.
Moreover, one can check that all Hamming balls of radius at most $t$ are contained in $\mc{C}_H$.
We now show that in $(U,\mc{C}_H)$, the Helly number is $2^{t+1}$ 
and the Radon number is $2^{t+1}+1$.
By the discussion above, we know $r(\mc{C}_H)>h(\mc{C}_H)\ge h(n,t;X)\ge 2^{t+1}$, where the last equality follows from an easy example in \cref{prop: lower bound for helly}.
So, it suffices to show that $r(\mc{C}_H) \le 2^{t+1}+1$.
To this end, let $p_1,p_2,\dots,p_m$ be $m \ge 2^{t+1}+1$ points in $X^n$.
Recall that for any set of points $P\subseteq X^n$, $\conv(P)$ is the intersection of the Hamming balls containing $P$, that is the intersection of $B(q,t)$s over all $q \in \bigcap_{p\in P} B(p,t)$.
By \cref{thm: general helly}, there exists $I\subseteq [m]$ of size 
at most $2^{t+1}$ so that $\bigcap_{i=1}^m B(p_i,t)=\bigcap_{i\in I} B(p_i,t)$.
This means $\emptyset\neq I \neq [m]$ and $\conv((p_i)_{i=1}^m)=\conv((p_i)_{i\in I})$.
Hence, $\emptyset\neq\conv((p_i)_{i\notin I})\subseteq \conv((p_i)_{i=1}^m)=\conv((p_i)_{i\in I})$, i.e. $\conv((p_i)_{i\notin I})\cap\conv((p_i)_{i\in I})\neq \emptyset$.
This proves $r(\mc{C}_H)\le 2^{t+1}+1$, as needed.

In the original setting of convex sets in $\mb{R}^d$, the following two extensions of Helly's theorem received a considerable amount of attention. The fractional Helly theorem, first proved by Katchalski and Liu~\cite{KL79}, states that in a finite family of convex sets in $\mb{R}^d$, if an $\alpha$-fraction of the $(d+1)$-tuples of sets in this family intersect, then one can select a $\beta$-fraction of the sets in the family with a nonempty intersection.
The Hadwiger-Debrunner conjecture, also known as the $(p,q)$-theorem, was first proved by Alon and Kleitman~\cite{AK92}.
It states that for $p\ge q\ge d+1$, if among any $p$ convex sets in the family, $q$ of them intersect, then there is a set of $O_{d,p,q}(1)$ points in $\mb{R}^d$ such that every convex set in the family contains at least one of these points. 
See also \cite{BK22} for more recent variants and extensions.

One can ask for versions of fractional Helly and $(p,q)$ theorems in general convexity spaces as well. Moreover, it is known that finite Radon number implies the fractional Helly theorem~\cite{HL21} and in turn the fractional Helly theorem implies the $(p,q)$ theorem~\cite{AKMM02}, both for suitable parameters.
In the case of Hamming balls of radius $t$, one can use these general results together with the fact that $r(\mc{C}_H)=2^{t+1}+1$ to obtain the fractional Helly theorem, where $\ell$-tuples of Hamming balls are considered with $\ell \gg 2^{t+1}$, and the $(p,q)$-theorem where $p>q \ge \ell \gg 2^{t+1}$.
Such resulst are very far from being optimal and instead we will prove them directly with much better dependencies on $t$. Interestingly, for both of them, if $|X|=2$, we only need the information on pairs of Hamming balls. On the other hand, if $|X|=\infty$, we need the information on $(t+2)$-tuples of Hamming balls. In particular, the threshold to have both theorems for Hamming balls (of radius $t$) is either 2 or $t+2$, much smaller than the corresponding Helly number. This is very different from convex sets in $\mb{R}^d$, where the threshold to have both theorems is $d+1$, the same as the corresponding Helly number.

\subsection{Algebraic tools and set-pair inequalities}

The proof of \cref{thm: general helly} is based on a novel variant of the so-called dimension argument. 
Surprisingly, this variant allows us to prove some upper bounds that do not depend on the dimension of the ambient space.
We believe that this may have further applications. 
For the special case of binary strings, that is, $|X|=2$, we prove a stronger statement by a probabilistic argument.
For convenience, we define the following two functions $f(t,X)$ and $f'(t,X)$.
\begin{definition}\label{def: f and f'}
    Let $t \ge 0$ and $X$ be any set of cardinality $|X| \ge 2$.
    Define 
    \begin{itemize}
        \item $f(t;X)$ to be the maximum $m$ such that there exists $n>t$ 
		and $a_1,a_2,\dots,a_m,b_1,b_2,\dots,b_m\in X^n$ where $\dist(a_i,b_i)\ge t+1$ for all $i \in [m]$ and $\dist(a_i,b_j) \le t$ for all distinct $i,j\in [m]$;
        \item $f'(t;X)$ to be the maximum $m$ such that there exists $n>t$ and $a_1,a_2,\dots,a_m,b_1,b_2,\dots,b_m\in X^n$ where $\dist(a_i,b_i)\ge t+1$ for all $i\in [m]$ and $\dist(a_i,b_j)+\dist(a_j,b_i) \le 2t$ for all distinct $i,j\in [m]$.
    \end{itemize}
\end{definition}
The study of these functions can also be motivated by the well-known 
set-pair inequalities in extremal set theory.
% They can be seen as a new variant of the so-called ``set-pair method''.
The set-pair inequalities, initiated by Bollob\'as~\cite{Bollobas65}, 
play an important role in extremal combinatorics with applications
in the study of saturated (hyper)-graphs, $\tau$-critical hypergraphs, 
matching-critical hypergraphs, and more. See \cite{Tuza94,Tuza96}
for surveys.
A significant generalization of Bollob\'as' result
is due to F\"uredi~\cite{Furedi84}.
It states that if $A_1,A_2,\dots,A_m$ are sets of size $a$ and $B_1,B_2,\dots,B_m$ are sets of size $b$ such that $|A_i\cap B_i|\le k$ for all $i\in[m]$ and $|A_i\cap B_j|>k$ for $1\le i<j\le m$, then $m \le \binom{a+b-2k}{a-k}$, and this is tight.
Using this result, one can give a short argument that $f(t;X)$ is finite.
\begin{proposition}
    $f(t;X) \le \binom{2t+2}{t+1}$ for every $t \ge 0$ and every set $X$.
\end{proposition}
\begin{proof}
    Suppose, for some $n > t$, that there are $a_1,a_2,\dots, a_m,b_1,b_2,\dots,b_m \in X^n$ satisfying $\dist(a_i,b_i)$ $\ge t+1$ for all $i$ and $\dist(a_i,b_j) \le t$ for all distinct $i,j\in[m]$.
    For each $i\in [m]$, let $A_i:=\{(k,a_k): k=1,2,\dots,n\}$ and $B_i:=\{(k,b_k): k=1,2,\dots,n\}$ be sets of $n$ pairs.
    Observe that $|A_i\cap B_j|+\dist(a_i,b_j)=n$ for all $i,j\in[m]$.
    It holds that $|A_i\cap B_i|\le n-t-1$ for all $i$ and $|A_i\cap B_j|\ge n-t$ for all distinct $i,j$.
     Since $|A_i|=|B_i|=n$, the above result of F\"uredi implies $m \le \binom{2n-2(n-t-1)}{n-(n-t-1)}=\binom{2t+2}{t+1}$, as desired.
\end{proof}

We note that $f(t;X) \le f'(t;X)$ holds by definition and that any upper bound on $f(t;X)$ implies the corresponding bound in \cref{thm: general helly}. 
Indeed, suppose $\mc{B}=\{B_1,B_2,\dots,B_m\}$ is a minimal collection 
of Hamming balls in $X^n$ of radius $t$ such for any $I\subseteq [m]$ of size at most $f(t;X)$, $\bigcap_{i=1}^m B_i\neq \bigcap_{i\in I} B_i$.
This means $m>f(t;X)$ and $\bigcap_{i} B_i\neq \bigcap_{i\neq j} B_i$ holds for all $j$ (using that $\mc{B}$ is minimal).
So, for each $j \in [m]$, there exists $b_j \in \bigcap_{i\neq j}B_i\setminus \bigcap_i B_i$.
In addition, let $a_j$ be the center of $B_j$.
Then, $\dist(a_i,b_i) \ge t+1$ for all $i$ and $\dist(a_i,b_j) \le t$ for all distinct $i,j$.
Hence, $m \le f(t;X)$, contradicting that $m > f(t;X)$.
This argument proves \cref{thm: general helly} with the help of the following result.
% For general $X$, the Helly-type result for Hamming balls (\cref{t11}) follows from the following.
\begin{theorem}\label{thm: main result for f}
    % Let $n > t \ge 0$ and $X$ be any set of size at least 2.
    $f(t;X)=2^{t+1}$ for every $t \ge 0$ and every set $X$ with $|X| \ge 2$.
\end{theorem}
In the binary case, we further prove the following.
\begin{theorem} \label{thm: main result for f'}
    $f'(t;\{0,1\}) = 2^{t+1}$ for every $t \ge 0$.
\end{theorem}

Our proof of \cref{thm: main result for f,thm: main result for f'} works in the more general setting where we assume $\dist(a_i,b_i) \ge t+s$ (for some $s \ge 1$) instead of $\dist(a_i,b_i) \ge t+1$.
For simplicity, we denote $f(t,s;X)$ and $f'(t,s;X)$ as the corresponding families' largest size.
Precisely, our proof shows that $f(t,s;X)\le {2^{t+s}}/{V_{t+s,s}}$ and $f'(t,s;\{0,1\})\le {2^{t+s}}/{V_{t+s,s}}$, where 
\begin{equation} \label{eq: V}
    V_{n,d} := \left\{
        \begin{array}{lll}
            \sum_{i=0}^{(d-1)/2} \binom{n}{i} & & d \text{ is odd} \\
            \sum_{i=0}^{d/2-1}\binom{n}{i}+\binom{n-1}{d/2-1} & & d \text{ is even}
        \end{array}
    \right.
\end{equation}
We note that $V_{n,d}$ is the size of the Hamming ball in $\{0,1\}^n$ of radius $\frac{d-1}{2}$ if $d$ is odd and of the union of two Hamming balls in $\{0,1\}^n$ of radius $\frac{d}{2}-1$ whose centers are of Hamming distance 1 if $d$ is even.
Interestingly, $V_{n,d}$ is also known to be the maximum possible 
cardinality of a set of points of diameter at most $d-1$ in $\{0,1\}^n$;
see \cite{Katona64,Kleitman66,Bezrukov87}.
We also note that when $d$ is odd, $2^n/V_{n,d}$ is the well-known 
Hamming bound for the maximum possible number of codewords in a 
binary {\em error correcting code} (ECC) of length $n$ and 
distance $d$.
Binary ECCs, which are large collections of binary strings with a 
prescribed minimum Hamming distance between any pair,
are widely studied and applied in computing, telecommunication, 
information theory and more; see  \cite{MS77,MMSS77}.
Indeed, ECCs naturally define $a_i$s and $b_i$s in \cref{def: f and f'}.
As we will show in \cref{sec: binary string}, the existence of ECCs 
that match the Hamming bound (the so-called {\em perfect codes}) 
and their extensions imply that 
$f(t,s;X)=f'(t,s;\{0,1\})={2^{t+s}}/{V_{t+s,s}}$ when 
$s\in\{1,2\}$, or $s\in\{3,4\}$ and $t+4$ is a power of $2$, 
or $s\in\{7,8\}$ and $t=16$.
This will be shown using the well-known Hamming code and the Golay code.
In addition, the famous BCH codes discovered by
Bose, Chaudhuri and Hocquenghem 
imply that our bounds are close to being tight when $s$ is fixed, 
that is, $f(t,s;X)=\Theta_s(2^{t+s}/V_{t+s,s})$ 
and $f'(t,s;X)=\Theta_s(2^{t+s}/V_{t+s,s})$.

Another well-known result in extremal set theory due to Tuza~\cite{Tuza87} states that if $(A_i,B_i)_{i=1}^m$ satisfies $A_i\cap B_i=\emptyset$ for $i\in[m]$ and $(A_i\cap B_j)\cup(A_j\cap B_i)\neq\emptyset$ for distinct $i,j\in[m]$, then $\sum_{i=1}^m p^{|A_i|}(1-p)^{|B_i|} \le 1$ for all $0<p<1$.
This also has various applications; see \cite{Tuza94,Tuza96}.
When $|A_i|+|B_i|=t+1$ for all $i$, this result implies $m \le 2^{t+1}$, which is tight.
\cref{thm: main result for f'} generalizes this by taking $A_i:=\{k\in [n]: a_{i,k}=1\}$ and $B_i:=\{k\in[n]: b_{i,k}=1\}$: if $|A_i\triangle B_i| \ge t+1$ for all $i \in [m]$ and $|A_i\triangle B_j|+|A_j\triangle B_i| \le 2t$ for all distinct $i,j \in [m]$, then $m \le f'(n,t,\{0,1\})=2^{t+1}$. Here, 
we do not require $A_i$ and $B_i$ to be disjoint and 
write $A\triangle B:=(A\setminus B)\cup (B\setminus A)$ for the symmetric difference of $A$ and $B$.

Finally, we mention briefly that \cref{thm: main result for f} 
motivates the study of a natural variant of 
the {\em Prague dimension} (also called the {\em product dimension})
of graphs. Initiated by Ne\v{s}et\v{r}il, Pultr and R\"odl~\cite{NP77,NR78}, the Prague dimension of a graph is the minimum $d$ such that every vertex is uniquely mapped to $\mb{Z}^d$ and two vertices are connected by an edge if and only if the corresponding vectors differ in all coordinates, i.e, it is
the minimum possible number of proper vertex colorings of $G$ so that for every pair $u, v$ of non-adjacent vertices there is at least one coloring in which $u$ and $v$ have the same color.
This notion has been studied intensively, see, e.g., 
\cite{LNP80,Alon86,ER96,Furedi00,AA20,GPW23}.

The rest of this paper is organized as follows.
In \cref{sec: general string} we present the proof of the main result \cref{thm: main result for f}. 
\cref{sec: binary string} deals with binary strings and briefly discusses the behavior of $f'(n,t,X)$ for $|X|>2$. 
We also discuss the connection to error correcting codes and another set-pair inequality.
In \cref{sec: related questions} we investigate several variants and generalizations of the main results including a fractional Helly-type result, a Hadwiger-Debrunner-type result, a variant of the Prague dimension of graphs, and a generalization of $f(t;X)$ to sequences of sets.
The final \cref{sec: concluding} contains some concluding remarks and open problems.

\section{General strings} \label{sec: general string}
We start by describing the lower bound of \cref{t11}, given by~\cite{RST23}.
This also provides the lower bounds in \cref{thm: main result for f,thm: main result for f'} as $f'(t;X)\ge f(t;X) \ge h(n,t;X) \ge 2^{t+1}$
\begin{proposition} \label{prop: lower bound for helly}
    $h(n,t;X) \ge 2^{t+1}$ for any $n>t$ and any set $X$ of cardinality $|X| \ge 2$.
\end{proposition}
\begin{proof}
    We may assume $n=t+1$ as $h(n,t;X) \ge h(t+1,t;X)$ and that $0,1 \in X$.
    Consider all the $2^{t+1}$ Hamming balls $B(a,t)$ where $a \in \{0,1\}^{t+1}$.
    It suffices to show that any $2^{t+1}-1$ balls intersect while all of them do not.
    Observe that $B(a,t)=\{0,1\}^{t+1}\setminus \{\overline{a}\}$ where $\overline{a} \in \{0,1\}^{t+1}$ is given by flipping all coordinates of $a$, i.e. changing 0 to 1 and changing 1 to 0.
    % Thus, $B(a,t)=\{0,1\}^{t+1}\setminus \{\overline{a}\}$.
    Therefore, for any $\ell=2^{t+1}-1$ vectors $a_1,\dots,a_\ell \in \{0,1\}^{t+1}$, the intersection $\bigcap_{i=1}^\ell B(a_i,t)$ contains all but at most $\ell < 2^{t+1}$ elements in $\{0,1\}^{t+1}$. 
    In other words, any $2^{t+1}-1$ such Hamming balls intersect.
    On the other hand, $\bigcap_{a\in \{0,1\}^{t+1}} B(a,t)=\bigcap_{a\in \{0,1\}^{t+1}}(\{0,1\}^{t+1}\setminus \{\overline{a}\})=\{0,1\}^{t+1}\setminus \bigcup_{a\in \{0,1\}^{t+1}} \{\overline{a}\}=\emptyset$, i.e. all the $2^{t+1}$ balls do not intersect.
\end{proof}

The rest of this section contains the proof of the upper bound 
of \cref{thm: main result for f}, and thus also of \cref{t11}.
To this end, we need the following properties of $V_{n,d}$.
\begin{claim}\label{claim: volume properties}
    $V_{n,d}\ge 2V_{n-1,d-1}$ for $2 \le d\le n$ and this is an equality if $d$ is even.
    In particular, $V_{n,d}\ge 2^{d-1}$ for all $1 \le d \le n$.
\end{claim}
\begin{proof}
    If $d=2k$ for some $k \in \{1,2,\dots,\floor{\frac{n}{2}}\}$, then 
    \begin{equation} \nonumber
        V_{n,d}
        = \sum_{i=0}^{k-1}\binom{n}{i}+\binom{n-1}{k-1}
        = \sum_{i=0}^{k-1}\binom{n-1}{i}+\sum_{i=0}^{k-2}\binom{n-1}{i}+\binom{n-1}{k-1}
        = 2 \sum_{i=0}^{k-1}\binom{n-1}{i}
        = 2V_{n-1,d-1}.
    \end{equation}
    If $d=2k-1$ for some $k \in \{2,3,\dots,\floor{\frac{n+1}{2}}\}$, then 
    \begin{equation} \nonumber
        V_{n,d}
        = \sum_{i=0}^{k-1}\binom{n}{i}
        = 2 \sum_{i=0}^{k-2}\binom{n-1}{i} + \binom{n-1}{k-1}
        \ge 2 \sum_{i=0}^{k-2}\binom{n-1}{i} + 2\binom{n-2}{k-2}
        = 2V_{n-1,d-1}.
    \end{equation}
    Here, we used $\binom{n-1}{k-1}\ge 2\binom{n-2}{k-2}$ as $k \le \frac{n+1}{2}$.

    Given the first inequality, $V_{n,d}\ge 2V_{n-1,d-1}\ge\dots\ge 2^{d-1}V_{n-d+1,1}=2^{d-1}$, as desired.
\end{proof}

We now show the upper bound of \cref{thm: main result for f} by the following stronger theorem.
\begin{theorem} \label{thm: general theorem for f}
    Let $n > t \ge 0, m \ge 1$, and $X$ be nonempty.
    Suppose $a_1,a_2,\dots,a_m,b_1,b_2,\dots,b_m \in X^n$, and assume that
	for each $i\in[m]$, $\dist(a_i,b_i)=t+s_i$ for some $s_i \ge 1$, 
	and $\dist(a_i,b_j) \le t$ for all distinct $i,j \in [m]$.
    Then, 
    \begin{equation} \label{eq: main eq for general X}
        \sum_{i=1}^m \frac{V_{t+s_i,s_i}}{2^{t+s_i}} \le 1.
    \end{equation}
    In particular, $f(t;X) \le 2^{t+1}$ and $f(t,s;X) \le 2^{t+s}/V_{t+s,s}$ if $s_i \ge s$ for all $i \in [m]$.
\end{theorem}
\begin{proof}
    First, suppose we have proved \cref{eq: main eq for general X}.
    Then, \cref{claim: volume properties} implies $1 \ge m\cdot {2^{s_i-1}}/{2^{t+s_i}}\ge m/2^{t+1}$, i.e. $m \le 2^{t+1}$.
    Hence, $f(t;X) \le 2^{t+1}$.
    Similarly, if $s_i \ge s$ for all $i \in [m]$, using \cref{claim: volume properties}, we acquire $1 \ge \sum_{i=1}^m \frac{V_{t+s_i,s_i}}{2^{t+s_i}}\ge m\cdot V_{t+s,s}/2^{t+s}$, i.e. $m \le 2^{t+s}/V_{t+s,s}$.
    So, $f(t,s;X) \le 2^{t+s}/V_{t+s,s}$.
    
\newcommand{\fii}{{f_{\scriptscriptstyle i,I_1,I_2}}}
\newcommand{\fjj}{{f_{\scriptscriptstyle j,J_1,J_2}}}
\newcommand{\fij}{{f_{\scriptscriptstyle i,J_1,J_2}}}
\newcommand{\xii}{{x_{\scriptscriptstyle i,I_1,I_2}}}
\newcommand{\cii}{{c_{\scriptscriptstyle i,I_1,I_2}}}
\newcommand{\cjj}{{c_{\scriptscriptstyle j,J_1,J_2}}}
    In the rest of the proof, we establish 
	\cref{eq: main eq for general X}.
    The proof is algebraic and uses a novel variant of the dimension argument which provides a dimension-free upper bound.
    Without loss of generality, assume $X \subseteq \mb{R}$.
    For each \( i \in [m] \), denote \( D_i:=\{k\in[n]: a_{i,k}\not = b_{i,k}\} \) and  $d_i$ to be the largest element in $D_i$.
    Then, \( |D_i|=\dist(a_i,b_i)=t+s_i \).
    In addition, we call a pair $(I_1,I_2)$ {\em compatible with $i$} if $I_1 \subseteq D_i, |I_1| \ge t+\frac{s_i+1}{2}, I_2 \subseteq [n]\setminus D_i$, or $I_1 \subseteq D_i\setminus\{d_i\}, |I_1|=t+\frac{s_i}{2}, I_2 \subseteq [n]\setminus D_i$
    (the latter happens only when $s_i$ is even).
    Note that $|I_1|\ge t+\frac{s_i}{2}$ in both cases and $|I_1|=t+\frac{s_i}{2}$ only if $I_1\subseteq D_i\setminus\{d_i\}$.
    For every $i \in [m]$ and every such pair $(I_1,I_2)$, define a polynomial on $x \in \mb{R}^n$ by 
    \begin{equation} \nonumber
        \fii(x) := \prod_{k \in I_1\cup I_2} (x_k-a_{i,k}) \prod_{k \in D_i\setminus I_1} (x_k-b_{i,k}).
    \end{equation}
    Recall \cref{eq: V}.
    The number of pairs compatible with $i$ is $V_{t+s_i,s_i}2^{n-(t+s_i)}$.
    Thus, it suffices to show that all such $\fii$s are linearly independent.
    Indeed, since every $\fii$ is a multilinear polynomial on $n$ variables, the linear independence implies $\sum_{i=1}^m  {V_{t+s_i,s_i}}{2^{n-(t+s_i)}}\le 2^n$, giving \cref{eq: main eq for general X}.

    To show the linear independence, we define, for each $i \in [m]$ and each $(I_1,I_2)$ compatible with $i$, an $x=\xii \in \mb{R}^n$ by $x_k=a_{i,k}$ for all $k \in D_i\setminus I_1$; $x_k=b_{i,k}$ for all $k \in I_1\cup ([n]\setminus (D_i\cup I_2))$; $x_k \in X \setminus \{b_{i,k}\}$ 
	arbitrary for all $k \in I_2$. 
    We also need the following ordering of the subsets of $[n]$: for distinct subsets $E,F \subseteq [n]$, we denote $E \prec F$ if $|E| < |F|$ or $|E|=|F|$ and $\max(E\setminus F)>\max(F\setminus E)$.
    We also write $E\preceq F$ if $E\prec F$ of $E=F$.
    It is easy to check that $\preceq$ induces a total order of all the subsets of $[n]$.    
    Now, we state the crucial claim for the evaluations of $\fii$s (on $\xii$s).
\newcommand{\pii}{{(D_i\setminus I_1)\cup I_2}}
\newcommand{\pjj}{{(D_j\setminus J_1)\cup J_2}}
    \begin{claim} \label{claim: evaluations}
        Let $i,j \in [m]$ and $(I_1,I_2)$ be compatible with $i$ and $(J_1,J_2)$ be compatible with $j$.
        Then, 
        \begin{enumerate}
            \item[(i)] for $i=j$, we have $\fjj(\xii) \neq 0$ if and only if $I_1=J_1$ and $J_2\subseteq I_2$;
            % \item[(ii)] for $i\neq j$, we have $\fjj(\xii) \neq 0$ implies $|J_2|+\frac{s_j}{2}<|I_2|+\frac{s_i}{2}$, or $s_i=s_j=2$ and $|I_1|=|J_1|=t+1$ and $|I_2|=|J_2|$ and $I_2\cup\{d_{i}\}\prec J_2\cup\{d_j\}$.
            \item[(ii)] for $i \neq j$, we have $\fjj(\xii) \neq 0$ implies $\pjj \prec \pii$. %$|J_2|+\frac{s_j}{2}<|I_2|+\frac{s_i}{2}$, or $|J_2|+\frac{s_j}{2}=|I_2|+\frac{s_i}{2}$ and $(D_i\setminus I_1)\cup I_2 \prec (D_j\setminus J_1)\cup J_2$.
        \end{enumerate}
    \end{claim}
    \begin{proof}
        Write $x=\xii$ for simplicity.        
        First, when $i=j$, since $a_{i,k}=b_{i,k}$ for all $k \in J_2\subseteq [n]\setminus D_i$, 
        $$
            \fjj(\xii) = \fij(x)
            =\prod_{k \in J_1\cup J_2} (x_k-a_{i,k}) \prod_{k \in D_i\setminus J_1} (x_k-b_{i,k})
            =\prod_{k \in J_1} (x_k-a_{i,k})\!\!\!\! \prod_{k \in (D_i\setminus J_1)\cup J_2} (x_k-b_{i,k}).
        $$ 
        This means that $\fjj(x)\neq 0$ if and only if $x_k \neq a_{i,k}$ for all $k \in J_1$ and $x_k \neq b_{i,k}$ for all $k \in (D_i\setminus J_1)\cup J_2$.
        By the definition of $x=x_{i,I_1,I_2}$, we know that $x_k=a_{i,k}$ for all $k \in D_i\setminus I_1$ and $x_k=b_{i,k}$ for $k \in I_1\cup ([n]\setminus (D_i\cup I_2))$.
        So, if $f_{j,J_1,J_2}(x) \neq 0$, then $(D_i\setminus I_1)\cap J_1=\emptyset$, $I_1\cap (D_i\setminus J_1)=\emptyset$ and $([n]\setminus (D_i\cup I_2)) \cap J_2=\emptyset$, i.e. $I_1=J_1$ and $J_2\subseteq I_2$.
        On the other hand, if $I_1=J_1$ and $J_2\subseteq I_2$, using that $x_k \neq b_{i,k}$ for all $k \in I_2$, it is easy to see that $\fjj(x) \neq 0$.
        This demonstrates (i).

        For (ii), suppose $\fjj(x)\neq 0$. The goal is to show $\pjj \prec \pii$.
        The fact that $$\fjj(x)=\prod_{k \in J_1\cup J_2} (x_k-a_{j,k}) \prod_{k \in D_j\setminus J_1} (x_k-b_{j,k}) \neq 0$$
        implies $x_k \neq a_{j,k}$ for all $k \in J_1\cup J_2$.
        In particular, since $x_k=b_{i,k}$ for all $k \in I:=I_1\cup ([n]\setminus (D_i\cup I_2))$ (by the definition of $x=x_{i,I_1,I_2}$), it holds that  $b_{i,k}\neq a_{j,k}$ for all $k \in (J_1\cup J_2)\cap I$, meaning $\dist(b_i,a_j)\ge |(J_1\cup J_2)\cap I|$.
        Then, the assumption $\dist(b_i,a_j) \le t$ implies $|(J_1\cup J_2)\cap I| \le t$.
        Observe that $[n]\setminus I=\pii$, and thus 
        \begin{equation} \label{eq: raw}
            |J_1|+|J_2| = |J_1\cup J_2|
            = |(J_1\cup J_2)\cap I|+|(J_1\cup J_2)\cap ([n]\setminus I)|
            \le t + |[n]\setminus I|
            = t + \abs{\pii}.
        \end{equation}
        Namely, $|J_2| \le t +\abs{\pii} - |J_1|$.
        Then, using $|D_j|=t+s_j$ and $|J_1| \ge t+s_j/2$, we obtain
        \begin{equation}  \label{eq: order by size}
            \abs{\pjj}
            =  |D_j|-|J_1|+|J_2|
            \le  (t+s_j)-2|J_1| + t + \abs{\pii}
            \le \abs{\pii}.
        \end{equation}
        If $\abs{\pjj} < \abs{\pii}$, then $\pjj \prec \pii$, and we are done.
        
        From now on, let us assume that $\abs{\pjj} = \abs{\pii}$.
        For simplicity, write $E:=\pjj$ and $F:=\pii$.
        As $|E|=|F|$, our goal is to show that $E\neq F$ and $\max(E\setminus F)>\max(F\setminus E)$.
        Note that the derivation of \cref{eq: order by size} demonstrates that $|E|=|F|$ only if \cref{eq: raw} is an equality and that $|J_1|=t+s_j/2$.
        In particular, the former implies that $|(J_1\cup J_2)\cap ([n]\setminus I)|=|[n]\setminus I|$, i.e. $F=\pii = [n] \setminus I \subseteq J_1 \cup J_2$; the latter implies $J_1 \subseteq D_j\setminus \{d_j\}$.
        Hence, $d_j \notin F$ while $d_j \in \pjj=E$, showing $E\neq F$.
        In addition, $\max(E \setminus F) \ge d_j$ as $d_j \in E \setminus F$.
        Suppose for contradiction that $\max(F \setminus E)>\max(E \setminus F)$ ($\ge d_j$).
        % We have $\max(F \setminus E) > d_j$ ($:=\max(D_j)$).
        However, since $F \subseteq J_1 \cup J_2 \subseteq (D_j\setminus \{d_j\})\cup J_2$ and $d_j=\max(D_j)$, we have $\max(F\setminus E) \in J_2\subseteq E$.
        This is impossible, so $\max(E\setminus F)>\max(F\setminus E)$ must hold.
        This shows $\pjj=E\prec F=\pii$, as desired.
    \end{proof}
    
    We now complete the proof by showing that all the $\fii$s constructed for $i\in[m]$ and $(I_1,I_2)$ compatible with $i$ are linearly independent.
    Suppose that $F:=\sum_{\scriptscriptstyle (j,J_1,J_2)} \cjj \fjj$ is the zero polynomial, where $\cjj \in \mb{R}$ for $j \in [m]$ and $(J_1,J_2)$ compatible with $j$.
    It suffices to prove $\cjj=0$ for all $(j,J_1,J_2)$.
    If not, we pick a triple $(i,I_1,I_2)$ with $\cii\neq 0$; if there are multiple such $(i,I_1,I_2)$s, pick the one that minimizes $\pii$ in the total order $\preceq$; if there is still a tie, then pick any of them.
    Consider $x:=\xii$, and suppose that $\cjj\fjj(x)\neq 0$ for some $(j,J_1,J_2)$.
    If $i=j$, then \cref{claim: evaluations}(i) implies $J_1=I_1$ and $J_2 \subseteq I_2$.
    Due to the minimality of $\pii$, it must be that $J_2=I_2$, meaning $(i,I_1,I_2)=(j,J_1,J_2)$.
    If $i \neq j$, \cref{claim: evaluations}(ii) implies that $\pjj \prec \pii$.
    Again, the minimality of $\pii$ implies $\cjj=0$. But this is impossible as we assumed $\cjj\fjj(x)\neq 0$.
    Altogether, $\cjj\fjj(x) \neq 0$ implies that $(i,I_1,I_2)=(j,J_1,J_2)$.
    In addition, \cref{claim: evaluations}(i) asserts $\fii(x) \neq 0$.
    So, $0=F(x)=\sum_{\scriptscriptstyle (j,J_1,J_2)} \cjj \fjj(x)=\cii\fii(x)$, and thus $\cii=0$.
    This contradicts our assumption that $\cii \neq 0$.
    Therefore, $\cii=0$ for all $(i,I_1,I_2)$, and this shows that all the $\fii$s are linearly independent.
    % \begin{lemma}
    %     All $\fii$s constructed for $i\in[m]$ and $(I_1,I_2)$ compatible with $i$ are linearly independent.
    % \end{lemma}
    % \begin{proof}
    %     Suppose that $F:=\sum_{\scriptscriptstyle (j,J_1,J_2)} \cjj \fjj$ is the zero polynomial, where $\cjj \in \mb{R}$ for $j \in [m]$ and $(J_1,J_2)$ compatible with $j$.
    %     Our goal is to prove $\cjj=0$ for all $(j,J_1,J_2)$.
    %     Suppose not. 
    %     Pick a triple $(i,I_1,I_2)$ with $\cii\neq 0$; if there are multiple such $(i,I_1,I_2)$s, pick the one that minimizes $\pii$ with respect to the total order $\preceq$; if there is still a tie, then pick any of them.
    %     Consider $x:=\xii$, and suppose that $\cjj\fjj(x)\neq 0$ for some $(j,J_1,J_2)$.
    %     If $i=j$, then \cref{claim: evaluations}(i) implies $J_1=I_1$ and $J_2 \subseteq I_2$.
    %     Due to the minimality of $\pii$, it must be that $J_2=I_2$, meaning $(i,I_1,I_2)=(j,J_1,J_2)$.
    %     If $i \neq j$, \cref{claim: evaluations}(ii) implies that $\pjj \prec \pii$.
    %     Again, the minimality of $\pii$ implies $\cjj=0$. But this is impossible as we assumed $\cjj\fjj(x)\neq 0$.
    %     Altogether, $\cjj\fjj(x) \neq 0$ implies that $(i,I_1,I_2)=(j,J_1,J_2)$.
    %     Also, \cref{claim: evaluations}(i) asserts $\fii(x) \neq 0$.
    %     So, $0=F(x)=\sum_{\scriptscriptstyle (j,J_1,J_2)} \cjj \fjj(x)=\cii\fii(x)$, i.e. $\cii=0$.
    %     This contradicts our assumption that $\cii \neq 0$.
    %     Therefore, $\cii=0$ for all $(i,I_1,I_2)$, and this shows that all $\fii$s are linearly independent.
    % \end{proof}
\end{proof}

% We conclude this section by showing that \cref{t11} extends to Hamming balls of radius {\em at most $t$}.
% \begin{corollary}
%     Let $n>t \ge 0$ and $X$ be any set of cardinality $|X|\ge 2$.
%     The Helly number of the family of all Hamming balls of radius at most $t$ in $X^n$ is exactly $2^{t+1}$.
% \end{corollary}
% \begin{proof}
%     Let $h$ be this Helly number.
%     It suffices to show $h \le 2^{t+1}$ as the lower bound follows from \cref{prop: lower bound for helly}.
%     We may assume that $0,1\in X$.
%     By the definition of $h$, there exists a minimal family of Hamming balls $\mc{B}=\{B_1,B_2,\dots,B_m\}$ in $X^n$ of radius at most $t$ such that every $h-1$ of them intersect while all of them do not.
%     Then, $m \ge h$, and every $m-1$ of these balls intersect (as $\mc{B}$ is minimal).
%     For each $B_i$, let $a_i\in X^n$ be its center and $r_i \in [t]$ be its radius, and $b_i \in \bigcap_{j\neq i} B_j\setminus B_i$.
%     Thus, $\dist(a_i,b_i)>r_i$ for $i \in [m]$ and $\dist(a_i,b_j) \le r_i$ for $i \neq j$.
%     Now, for $1 \le i \le m$, we create $(t-r_i)$ more coordinates to all the $2m$ vectors by appending $(t-r_i)$ 1s to $a_i$ and $(t-r_i)$ 0s to all others.
%     In the end, we obtain $a_1',a_2',\dots,a_m',b_1',b_2',\dots,b_m'\in X^{n'}$ where $n'=n+\sum_{i=1}^m (t-r_i)$ such that $\dist(a_i',b_j')=\dist(a_i,b_j)+t-r_i$ for all $1 \le i,j \le m$.
%     This means $\dist(a_i',b_i')>t$ for $i \in [m]$ and $\dist(a_i',b_j')\le t$ for $i\neq j$.
%     So, $h\le m \le 2^{t+1}$ by \cref{thm: general theorem for f}, as desired.
% \end{proof}

\section{Binary strings} \label{sec: binary string}
This section deals with the binary setting, e.g. $X=\{0,1\}$.
In this case, we can prove a stronger result (\cref{thm: main result for f'}) where we only assume that $\dist(a_i,b_j)+\dist(a_j,b_i)\le 2t$ for $i\neq j$.
The lower bound is, again, derived from \cref{prop: lower bound for helly}.
For the upper bound, we provide a probabilistic proof that is simpler than that of \cref{thm: general theorem for f}.
\begin{theorem} \label{thm: general theorem for f'}
    Let $n > t \ge 0$.
    Suppose $a_1,a_2,\dots,a_m,b_1,b_2,\dots,b_m \in \{0,1\}^n$ satisfy for all $i \in [m]$, $\dist(a_i,b_i) = t+s_i$ for some $s_i \ge 1$,  and $\dist(a_i,b_j)+\dist(a_j,b_i) \le 2t$ for all distinct $i,j\in[m]$.
    Then, 
    \begin{equation} \label{eq: main eq for binary case}
        \sum_{i=1}^m \frac{V_{t+s_i,s_i}}{2^{t+s_i}} \le 1.
    \end{equation}
    In particular, $f'(t;\{0,1\}) \le 2^{t+1}$ and $f'(t,s;X) \le 2^{t+s}/V_{t+s,s}$ if $s_i \ge s$ for all $i \in [m]$.
\end{theorem}
\begin{proof}
    Given \cref{eq: main eq for binary case}, the derivation of the bounds for $f'(t;\{0,1\})$ and $f'(t,s;\{0,1\})$ is the same as that in the proof of \cref{thm: general theorem for f}, so we omit it here.
    
    We have to prove \cref{eq: main eq for binary case}.
    For each $i$, denote $D_i:=\{k\in[n]: a_{i,k}\neq b_{i,k}\}$ and $d_i := \max(D_i)$.
    Then, $\abs{D_i}=\dist(a_i,b_i)=t+s_i$.
    Now, sample a string $\alpha$, uniformly in $\{0,1\}^n$.
    For each $i\in [m]$, let $D_i(\alpha):=\{k\in D_i: \alpha_k=a_{i,k}\}$.
    Denote $\mc{E}_i$ to be the event that either $|D_i(\alpha)| \ge t + \frac{s_i+1}{2}$, or $|D_i(\alpha)|=t+\frac{s_i}{2}$ and $d_i \notin D_i(\alpha)$ (the latter happens only when $s_i$ is even).
    In both cases, $|D_i(\alpha)| \ge t + \frac{s_i}{2}$.
    % By \cref{claim: volume properties}, $\Pr[\mc{E}_i]=\frac{V_{t+s_i,s_i}}{2^{t+s_i}}$.
    It suffices to establish \cref{claim: events are disjoint} below because then, $1 \ge \Pr[\mc{E}_1\cup\dots\cup\mc{E}_m]=\sum_{i=1}^m \Pr[\mc{E}_i]= \sum_{i=1}^n\frac{V_{t+s_i,s_i}}{2^{t+s_i}}$, using \cref{eq: V}.
    \begin{claim} \label{claim: events are disjoint}
	    The events
        $\mc{E}_1,\mc{E}_2,\dots,\mc{E}_m$ are pairwise disjoint.
    \end{claim}
    \begin{proof}
        Suppose for contradiction that $\mc{E}_i$ and $\mc{E}_j$ are not disjoint for some $1 \le i < j \le m$.
        Let $\alpha \in \mc{E}_i\cap \mc{E}_j$.
        For convenience, denote $D_{ij}:=D_i \setminus D_j$ and $D_{ji}:=D_j\setminus D_i$.
        By the definition of $D_i$, we know that $a_i$ and $b_i$ are the same restricted to $D_i^c$ while they are the opposite restricted to $D_i$, i.e. $\restrict{a_i}{D_i^c}=\restrict{b_i}{D_i^c}$ and $\restrict{a_i}{D_i}={\restrict{\overline{b_i}}{D_i}}$ (we use $\overline{\,\,\cdot\,\,}$ for the opposite string of the same index set).
        Similarly, $\restrict{a_j}{D_j^c}=\restrict{b_j}{D_j^c}$ and $\restrict{a_j}{D_j}={\restrict{\overline{b_j}}{D_j}}$.
        As a consequence, coordinates among $D_{ij}\cup D_{ji}$ will contribute to the distance $\dist(a_i,b_j)+\dist(b_i,a_j)$.
        More precisely, 
        \begin{equation} \label{eq: Di minus Dj}
            % \begin{aligned}
                \dist(\restrict{a_i}{D_{ij}},\restrict{b_j}{D_{ij}})+\dist(\restrict{b_i}{D_{ij}},\restrict{a_j}{D_{ij}}) 
                =\dist({\restrict{\overline{b_i}}{D_{ij}}},\restrict{b_j}{D_{ij}})+\dist(\restrict{b_i}{D_{ij}},\restrict{b_j}{D_{ij}}) = |D_{ij}|,
            % \end{aligned}
        \end{equation}
        and similarly, 
        \begin{equation} \label{eq: Dj minus Di}
            \dist(\restrict{a_i}{D_{ji}},\restrict{b_j}{D_{ji}}) + \dist(\restrict{b_i}{D_{ji}},\restrict{a_j}{D_{ji}})
            = |D_{ji}|.
        \end{equation}
        Write $g:=|D_i\cap D_j|$, and hence $|D_{ij}|=t+s_i-g,|D_{ji}|=t+s_j-g$.
        Then, \cref{eq: Di minus Dj,eq: Dj minus Di} imply
        \begin{equation} \label{eq: Di difference Dj}
            \begin{aligned}
                &\,\, \dist(a_i,b_j)+\dist(b_i,a_j) \\
                \ge&\,\, \dist(\restrict{a_i}{D_{ij}\cup D_{ji}},\restrict{b_j}{D_{ij}\cup D_{ji}})+\dist(\restrict{b_i}{D_{ij}\cup D_{ji}},\restrict{a_j}{D_{ij}\cup D_{ji}}) \\
                =&\,\, \dist(\restrict{a_i}{D_{ij}},\restrict{b_j}{D_{ij}})+\dist(\restrict{b_i}{D_{ij}},\restrict{a_j}{D_{ij}}) 
                    +\dist(\restrict{a_i}{D_{ji}},\restrict{b_j}{D_{ji}})+\dist(\restrict{b_i}{D_{ji}},\restrict{a_j}{D_{ji}}) \\
                =&\,\, |D_{ij}|+|D_{ji}|
                = 2t-2g + s_i+s_j.
            \end{aligned}
        \end{equation}
        By assumption, $2t \ge 2t-2g+s_i+s_j$, i.e. $2g \ge s_i + s_j$.

        The second step is to consider the contribution to $\dist(a_i,b_j)+\dist(b_i,a_j)$ from $k \in D_i\cap D_j$.
        Denote $D:=D_i(\alpha)\cap D_j(\alpha) \subseteq D_i \cap D_j$.
        Observe that $a_{i,k}=a_{j,k}=\alpha_k \neq b_{i,k}=b_{j,k}$ for $k \in D$. So
        \begin{equation} \label{eq: D}
            \dist(\restrict{a_i}{D},\restrict{b_j}{D})+\dist(\restrict{b_i}{D},\restrict{a_j}{D}) = 2|D|.
        \end{equation}
        We then lower bound $|D|$.
        Since $\alpha \in \mc{E}_i\cap \mc{E}_j$, $|D_i(\alpha)| \ge t+\frac{s_i}{2}$ and $|D_j(\alpha)| \ge t+\frac{s_j}{2}$.
        Writing $D_i':=D_i(\alpha)\cap (D_i\cap D_j)$ and $D_j':=D_j(\alpha)\cap (D_i\cap D_j)$, we know 
        \begin{equation}  \label{eq: Di' and Dj'}
            \left\{
            \begin{aligned}
                &|D_i'|=|D_i(\alpha)\cap (D_i\cap D_j)|
                = |D_i(\alpha)\setminus D_{ij}|
                \ge t + \frac{s_i}{2} - (t + s_i - g)
                = g - \frac{s_i}{2} \\ 
                &|D_j'|=|D_j(\alpha)\cap (D_i\cap D_j)|
                = |D_j(\alpha)\setminus D_{ji}|
                \ge t + \frac{s_j}{2} - (t + s_j - g)
                = g - \frac{s_j}{2},
            \end{aligned}        
            \right.
        \end{equation}
        and thus
        \begin{equation} \label{eq: size of D}
            \begin{aligned}
                |D| &= |D_i'\cap D_j'|
                = |D_i'| + |D_j'| - |D_i'\cup D_j'| \\
                &\ge |D_i'|+|D_j'|-|D_i\cap D_j|
                \ge g-\frac{s_i}{2}+g-\frac{s_j}{2} - g
                = g-\frac{s_i+s_j}{2}.
            \end{aligned}
        \end{equation}
        We note that the RHS of \cref{eq: size of D} is non-negative because  $2g \ge s_i + s_j$.
        Using \cref{eq: Di difference Dj,eq: Di minus Dj,eq: Dj minus Di,eq: D,eq: size of D},
        \begin{equation} \nonumber
            \begin{aligned}
                2t \,\ge&\, \dist(a_i,b_j)+\dist(b_i,a_j) \\
                \ge&\, \dist(\restrict{a_i}{D_{ij}\cup D_{ji}},\restrict{b_j}{D_{ij}\cup D_{ji}})+\dist(\restrict{b_i}{D_{ij}\cup D_{ji}},\restrict{a_j}{D_{ij}\cup D_{ji}}) + \dist(\restrict{a_i}{D},\restrict{b_j}{D})+\dist(\restrict{b_i}{D},\restrict{a_j}{D}) \\
                \ge&\, |D_{ij}|+|D_{ji}| + 2|D|
                = (t+s_i-g) + (t + s_j - g) + 2g - (s_i+s_j)
                = 2t.
            \end{aligned}
        \end{equation}
        This being an equality implies, in particular, \cref{eq: size of D} is an equality, so \cref{eq: Di' and Dj'} is also an equality.
        The former means $D_i'\cup D_j'=D_i\cap D_j$ while the latter means $D_{ij} \subseteq D_i(\alpha), D_{ji} \subseteq D_j(\alpha)$ and $|D_i(\alpha)|=t+\frac{s_i}{2},|D_j(\alpha)|=t+\frac{s_j}{2}$.
        By the definition of $\mc{E}_i$ and $\mc{E}_j$, $d_i \notin D_i(\alpha)$ and $d_j \notin D_j(\alpha)$, and thus, $d_i\notin D_i'$ and $d_j \notin D_j'$.
        Also, as $D_{ij} \subseteq D_i(\alpha)$ and $D_{ji} \subseteq D_j(\alpha)$, it must be that $d_i,d_j \in D_i\cap D_j$.
        Recall that $d_i=\max(D_i)$ and $d_j=\max(D_j)$.
        This means $d_i=d_j=\max(D_i\cap D_j)$.
        % However, this means that $d_i=d_j \in D_i\cap D_j$ but $d_i=d_j \notin D_i'\cup D_j'$.
        However, as discussed before, $d_i=d_j\notin D_i'\cup D_j'$.
        This contradicts that $D_i'\cup D_j'=D_i\cap D_j$.
        Therefore, $\mc{E}_i$ and $\mc{E}_j$ must be disjoint for all $1 \le i < j \le m$.
    \end{proof}
\end{proof}

Next, we discuss the tightness of \cref{thm: general theorem for f,thm: general theorem for f'}.
For $1 \le d \le n$, an $(n,d)$ error correcting code (ECC) is a 
collection of binary strings (codewords) of length $n$ with all
pairwise distances at least $d$.
Write $A(n,d)$ for the maximum possible size of such a collection.
Taking $n=t+s$, $a_i$ to be any one of the codewords and $b_i=\overline{a_i}$, it is easy to see that $f'(t,s,\{0,1\}) \ge A(t+s,s)$.
As discussed in the Introduction, the upper bound $f'(t,s;\{0,1\})\le \frac{2^{t+s}}{V_{t+s,s}}$ is the Hamming bound for ECC when $s$ is odd.
Thus, we can use perfect codes (ECCs that match the Hamming bound) and their extensions (add a parity bit so that the length and the distance increases by one while the number of codewords stays the same) to show our bound on $f'(t,s,\{0,1\})$ is tight, (and the same also holds for $f(t,s,X)$).
More precisely,
\begin{itemize}
    \item $f'(t,s,\{0,1\}=2^{t+1}$ for $s\in\{1,2\}$. We can take the trivial ECC, all the binary strings of length $t+1$. 
	    There are $2^{t+1}$ of them and all pairwise distances are at least 1. So, $f'(t,1,\{0,1\})=A(t+1,1)=2^{t+1}$. Adding a parity bit to all these strings, the pairwise distances are at least 2. So, $f'(t,2,\{0,1\})=A(t+2,2)=2^{t+1}$.
    \item $f'(t,s,\{0,1\}=\frac{2^{t+3}}{t+4}$ when $s\in\{3,4\}$ and $t+4$ is a power of 2. When $t+4$ is a power of 2, we take the Hamming code: $\frac{2^{t+3}}{t+4}$ binary strings of length $t+3$ and pairwise distances at least $3$. 
	    This shows that $f'(t,3,\{0,1\}=A(t+3,3)=\frac{2^{t+3}}{t+4}$. Adding a parity bit to all these strings, the pairwise distances are at least 4. So, $f'(t,4,\{0,1\})=A(t+4,4)=\frac{2^{t+3}}{t+4}$.
    \item $f'(16,7;\{0,1\})=f'(16,8;\{0,1\})=2048$. Here, we take the Golay code~\cite{Golay49}: 2048 binary strings of length 23 whose pairwise 
	    distances are at least 7. This, as wells as its extension, implies $f'(16,7;\{0,1\})=f'(16,8;\{0,1\})=2048$.
\end{itemize}
Besides the perfect codes, we can take the Bose–Chaudhuri–Hocquenghem codes (BCH codes)~\cite{BCH1,BCH2}. These are $\Omega({2^{t+s}}/{(t+s)^s})$ binary strings of length $t+s$ and pairwise distances at least $s$ whenever $s$ is odd.
Based on the former discussion, this, and its extension, demonstrate 
that for every fixed $s$, 
$f(t,s;X)=\Theta_s(2^{t+s}/V_{t+s,s})$ and $f'(t,s;X)=\Theta_s(2^{t+s}/V_{t+s,s})$.

We note that our probabilistic proof relies crucially 
on the fact that each coordinate is either 0 or 1.
A similar proof by sampling $\alpha \in X^n$ appropriately works for general $X$s but only gives an upper bound of $|X|^{t+1}$.
This is not merely a coincidence: when $|X| \in \{3,4\}$, unlike $f(t;X)=2^{t+1}$, we can prove that $f'(t;X)=\Theta(3^t)$.
\begin{theorem}
    % Let $X$ be a set of size 3 or 4, and $n \ge t + 1$. 
    $3^t \le f'(t;X) \le 3^{t+1}$ for every $t \ge 0$ and every $X$ of size 3 or 4.
\end{theorem}
\begin{proof}
    For the lower bound, we assume $\{0,1,2\} \subset X$.
    Define $\varphi:\{0,1,2\}\to \{0,1,2\}$ by $\varphi(0)=1,\varphi(1)=2$ and $\varphi(2)=0$.
    % Let $A$ be the set of $a \in X^{t+1}$ such that $a_k\in\{0,1,2\}$ for $1 \le k \le t+1$.
    % For each $s \in A$, denote $f(s)=(\varphi(s_1),\dots,\varphi(s_{t+1}),s_{t+2},\dots,s_n) \in A$.
    Let $a_1,a_2,\dots, a_m$ be an enumeration of $s \in \{0,1,2\}^{t+1}$ such that $\sum_{k=1}^{t+1} s_k$ is a multiple of 3. Since, for every choice of $s_1, \ldots, s_t$ there is  a unique 
    $s_{t+1}$ such that $\sum_{k=1}^{t+1} s_k$ is a multiple of 3, we have that $m=3^t$.
    For each $i \in [m]$, define $b_i\in\{0,1,2\}^{t+1}$ by $b_{i,k}=\varphi(a_{i,k})$ for all $k\in[t+1]$.
    Clearly $\dist(a_i,b_i)=t+1$ for all $i\in[m]$, and 
    for any $i\neq j$, it holds that 
    \begin{equation} \nonumber
        \begin{aligned}
            \dist(a_i,b_j)+\dist(a_j,b_i)
            =&\,|\{1 \le k \le t+1: a_{i,k}\neq \varphi(a_{j,k})\}|
                +|\{1 \le k \le t+1: \varphi(a_{i,k})\neq a_{j,k}\}| \\
            =&\, t+1 + \{k: a_{i,k}=a_{j,k}\}.
        \end{aligned}
    \end{equation}
    Since $\sum_{k=1}^{t+1} a_{i,k}$ and $\sum_{k=1}^{t+1} a_{j,k}$ are both multiples of 3, $a_i$ and $a_j$ can share at most $t+1-2=t-1$ bits. 
    Therefore, $ \dist(a_i,b_j)+\dist(a_j,b_i) \le t+1+t-1 = 2t$.
    This shows that $f'(t;X) \ge 3^t$.

    We now prove the upper bound.
    Suppose $n > t$ and $a_1,a_2,\dots,a_m,b_1,b_2\dots,b_m \in X^n$ with $\dist(a_i,b_i) \ge t+1$ for $i\in[m]$ and $\dist(a_i,b_j)+\dist(a_j,b_i) \le 2t$ for $i \neq j$.
    For $1 \le k \le n$, independently and uniformly sample a 2-element subset $X_k$ of $X$.
    For each $i \in [m]$, define $a_i'\in\{0,1\}^n$ by $a_{i,k}':=1$ iff $a_{i,k}\in X_i$ for all $k \in [n]$, and  $b_i'\in\{0,1\}^n$ by $b_{i,k}':=1$ iff $b_{i,k}\in X_i$ for all $k \in [n]$.   
    Denote $I:=\{i \in [m]: \dist(a_i',b_i') \ge t+1\}$.
    Using $|X|\in\{3,4\}$, it holds that $\Pr[a'_{i,k}\neq b'_{i,k}] = 2\binom{|X|-2}{1}/\binom{|X|}{2} = \frac{2}{3}$ whenever $a_{i,k}\neq b_{i,k}$.
    % For any $i \in [m]$ and $k \in [n]$ such that $a_{i,k}\neq b_{i,k}$, it holds that $\Pr[a'_{i,k}\neq b'_{i,k}] = 2\binom{|X|-2}{1}/\binom{|X|}{2} = \frac{2}{3}$ when $|X|\in\{3,4\}$.
    So, $\Pr[i \in I]\ge \left(\frac{2}{3}\right)^{t+1}$, and hence, $\EE|I| \ge m \left(\frac{2}{3} \right)^{t+1}$.
    Observe also that $\dist(a_i',b_j') \le \dist(a_i,b_j) \le t$ for any distinct $i,j\in I$, which means $|I| \le f'(t;\{0,1\}) \le 2^{t+1}$.
    Therefore, $m (\frac{2}{3})^{t+1} \le \EE|I| \le 2^{t+1}$, i.e. $m \le 3^{t+1}$.
\end{proof}
We remark that for general $X$, the same argument (by sampling $X_k \in \binom{X}{\floor{|X|/2}}$ instead) shows $3^{t} \le f'(n,t,X) \le \left(\frac{|X|(|X|-1)}{\floor{|X|/2}\ceil{|X|/2}}\right)^{t+1}$.
We do not know which of these bounds is closer to the truth.

\subsection{A set-pair result}
As mentioned in the Introduction, F\"uredi~\cite{Furedi84} proved that if $A_1,A_2,\dots,A_m$ are sets of size $a$ and $B_1,B_2,\dots,B_m$ are sets of size $b$ such that $|A_i\cap B_i|\le k$ for $i\in[m]$ and $|A_i\cap B_j|>k$ for distinct $i,j\in[m]$, then $m \le \binom{a+b-2k}{a-k}$, and this is tight.
In the same paper, he raised the question of understanding the largest 
possible size of a family $(A_i,B_i)_{i=1}^m$ such that $|A_i|=a, |B_i|=b, |A_i\cap B_i| \le \ell$ for all $i \in [m]$ and $|A_i\cap B_j|> k$,
(where $k \geq \ell$ are given) for all distinct $i,j \in [m]$.
His result shows that this maximum is exactly 
$\binom{a+b-2\ell}{a-\ell}$ in case $k=\ell$.
For the general case, Zhu~\cite{Zhu95} showed the answer is at most $\min(\binom{a+b-2\ell}{a-k}/\binom{a-\ell}{k-\ell},\binom{a+b-2\ell}{b-k}/\binom{b-\ell}{k-\ell})$, and this is tight if there is a collection $\mc{A}$ of subsets of $U:=[a+b-2\ell]$, each with size $a-\ell$, such that every subset of $U$ with size $a-k$ is contained in exactly one of $\mc{A}$,
or there is a collection $\mc{B}$ of subsets of $U$ with size $b-\ell$, such that every subset of $U$ with size $b-k$ is contained in exactly one of $\mc{B}$.
These collections are called {\em designs} or {\em Steiner systems} and 
exist when $a-\ell$ is sufficiently larger than $b-\ell$ or $b-\ell$ 
is sufficiently larger than $a-\ell$ provided the appropriate divisibility
conditions hold; see \cite{Keevash14,GKLO23}.

We note that with a slight change of his argument, we can show the answer is at most $\frac{\binom{a+b-2\ell}{a-\ell-x+y}}{\binom{a-\ell}{x}\binom{b-\ell}{y}}$ for every $x,y \ge 0$ with $x+y=k-\ell$.
This is $\exp(O(k-\ell))$ better than his original bound if, say, $a-\ell=b-\ell\gg k-\ell$. Indeed, 
by the general position and the dimension reduction arguments,
used in \cite{Furedi84,Zhu95}, we can essentially assume $\ell=0$ 
(with $a-\ell,b-\ell,k-\ell$ replacing $a,b,k$).
For each $i \in [m]$, we build $\binom{a}{x}\binom{b}{y}$ pairs of sets 
based on $(A_i,B_i)$ by shifting, in all possible ways, 
a subset $X$ of $x$ elements of $A_i$ from $A_i$ to $B_i$,
and a subset $Y$ of $y$ elements of $B_i$ from $B_i$ to $A_i$.
This gives $m\binom{a}{x}\binom{b}{y}$ pairs $(A_i^{\scriptscriptstyle{X,Y}},B_i^{\scriptscriptstyle{X,Y}})$ with $|A_i^{\scriptscriptstyle{X,Y}}|=a-x+y,|B_i^{\scriptscriptstyle{X,Y}}|=b-y+x,A_i^{\scriptscriptstyle{X,Y}}\cap B_i^{\scriptscriptstyle{X,Y}}=\emptyset$. We also have $|A_i^{\scriptscriptstyle{X,Y}}\cap B_i^{\scriptscriptstyle{X',Y'}}|>0$ if $X\neq X'$ and $|A_i^{\scriptscriptstyle{X,Y}}\cap B_j^{\scriptscriptstyle{X',Y'}}|>0$ if $i\neq j$, since $|X|\le k-|Y'|=|X'|$.
Now, we can apply the result of Bollob\'as~\cite{Bollobas65} to conclude that $m\binom{a}{x}\binom{b}{y}\le\binom{a+b}{a-x+y}$, as desired.

Moreover, \cref{thm: general theorem for f'} gives the following variation of F\"uredi's question where instead of $|A_i|=a,|B_i|=b$, we only require $|A_i|+|B_i|=s$.
\begin{theorem}
    Let $s > k \ge \ell \ge 0$ and $m \ge 0$.
    Suppose $A_1,A_2,\dots,A_m, B_1,B_2,\dots,B_m$ are sets such that $|A_i|+|B_i|=s$ for all $i\in[m]$, $|A_i\cap B_i| \le \ell$ for all $i \in [m]$ and $|A_i\cap B_j|+|A_j\cap B_i|\ge 2(k+1)$ for all $1 \le i < j \le m$.
    Then, $m \le f'(s-2(k+1),2(k+1)-2\ell;\{0,1\})\le \frac{2^{s-2\ell-1}}{\sum_{i=0}^{k-\ell}\binom{s-2\ell-1}{i}}$.
    % Indeed, we can relax $|A_i\cap B_j|\ge k+1$ by $|A_i\cap B_j|+|A_j\cap B_i|\ge 2(k+1)$.
\end{theorem}
\begin{proof}
    Suppose all sets are subsets of $[n]$ for some $n \in \mb{N}$.
    For $1 \le i \le m$, let $a_i\in\{0,1\}^n$ be the indicator vector of $A_i$, i.e. $a_{i,k}=1$ iff $k\in A_i$; similarly, let $b_i\in\{0,1\}^n$ be that of $B_i$.
    Then, for $i\neq j$,
    \begin{equation} \nonumber
        \left\{
        \begin{aligned}
            \dist(a_i,b_i)&=|A_i|+|B_i|-2|A_i\cap B_i|\ge s-2\ell,\\
            \dist(a_i,b_j)+\dist(a_j,b_i)
            &=|A_i|+|B_j|-2|A_i\cap B_j| + |A_j|+|B_i|-2|A_j\cap B_i|
            \le 2s-4(k+1).
        \end{aligned}
        \right.
    \end{equation}
    So, 
    $$
        m \le f'(s-2(k+1),2(k+1)-2\ell;\{0,1\})\le \frac{2^{s-2\ell}}{V_{s-2\ell,2(k+1)-2\ell}}
        = \frac{2^{s-2\ell-1}}{V_{s-2\ell-1,2(k+1)-2\ell-1}}
        = \frac{2^{s-2\ell-1}}{\sum_{i=0}^{k-\ell}\binom{s-2\ell-1}{i}}.
    $$
\end{proof}
Note that this bound is close to being tight when $s-2\ell \gg k-\ell$.
In this case, we can take the BCH code of length $s-2\ell-1$ and 
pairwise distances at least $2k-2\ell+1$.
Appending to each codeword a parity bit, we get $\Omega(2^{s-2\ell}/(s-2\ell)^{k-\ell})$ binary strings of length $s-2\ell$ and pairwise distances at least $2k-2\ell+2$.
Now, take $A_i\subseteq [s-2\ell]$ to be the set corresponding to each codeword joined with $\{-1,-2,\dots,-\ell\}$ and $B_i:=([s-2\ell]\setminus A_i) \cup \{-1,-2,\dots,-\ell\}$.
Then, $|A_i|+|B_i|=s,|A_i\cap B_i|=\ell$ for $i\in [m]$ and $|A_i\cap B_j|+|A_j\cap B_i|\ge 2\ell+2(k-\ell+1)=2k+2$ for $i\neq j$, forming the desired family.

\section{Related questions}\label{sec: related questions}

\subsection{Fractional-Helly-type and $(p,q)$-type problems for Hamming balls}
In this section, we  establish fractional-Helly and $(p,q)$ theorems for Hamming balls. For both of them, when $X$ is finite, we need only the information about pairs of Hamming balls, and that of $(t+2)$-tuples of Hamming balls when $X$ is infinite. Notably, both constants $2$ and $t+2$ are optimal. This is in strong contrast
with the bounds that can be obtained, using the general results~\cite{AKMM02,HL21}, which make use of the Radon number $r(\mc{C}_H)=2^{t+1}+1$. These results would imply 
fractional-Helly and $(p,q)$ theorems for Hamming balls where one requires the information abut $\ell$-tuples for $\ell> 2^{t+1}$.

We say a point {\em hits} a Hamming ball if the ball contains the point and a set of points {\em hits} a collection of Hamming balls if every ball contains some point in the set.

\begin{theorem} \label{thm: fractional helly}
    Let $m\ge 1,n > t \ge 0$, $X$ be any nonempty set and $B_1,\dots,B_m$ be Hamming balls of radius $t$ centered at $a_1,\dots,a_m \in X^n$, respectively.
    \begin{description}
        \item[(1)] If $X$ is finite and, for some $\alpha > \frac{12}{m}$, at least $\alpha\binom{m}{2}$ (unordered) pairs of the Hamming balls intersect, then some point in $X^n$ hits an $\Omega\!\left({\alpha^2}|X|^{-t}\big/{\binom{4t}{t}} \right)$ fraction of the balls.
        \item[(2)] If, for some $\alpha > 0$, at least $\alpha\binom{m}{t+2}$ unordered $(t+2)$-tuples (where tuples have distinct entries) of the Hamming balls intersect, then some point in $X^n$ hits an $\Omega(\alpha/(e(t+1))^{t+1})$ fraction of the balls.
    \end{description}
\end{theorem}
\begin{theorem}\label{thm: hadwiger-debrunner}
    Let $m \ge 1, n > t \ge 0, p \ge q \ge 2$, and $X$ be a nonempty set.
    Let $B_1,B_2,\dots, B_m$ be Hamming balls of radius $t$ centered at $a_1,a_2,\dots,a_m \in X^n$, respectively, where out of any $p$ balls, $q$ of them intersect.
    \begin{itemize}
        \item If $|X|<\infty$ and $q \leq t+1$, then there exist $p^qq^t|X|^{t+2-q}2^{O(t)}$ points in $X^n$ hitting all these Hamming balls;
        \item if $q=t+2$, then there exist $O(e^{2t}p^{t+1})$ points in $X^n$ hitting all these Hamming balls.
    \end{itemize}
\end{theorem}

\noindent
{\bf Remark.}\,
    Suppose $B_1,B_2,\dots,B_m$ are Hamming balls of radius $t$. If 
	every $t+2$ of them intersect, then there exist $O_t(1)$ points in $X^n$ hitting all of them (by \cref{thm: hadwiger-debrunner}). On the other hand the number 
	$t+2$ cannot be replaced by $t+1$, as can be seen by taking
	$n=t+1$ and an infinite $X$. 
	But if any $2^{t+1}$ of them intersect, then all of them intersect (\cref{t11}).
    This behavior differs from the setting of Helly's theorem, where a family of convex sets in $\mb{R}^d$ is considered: if every $d+1$ of them intersect, then all intersect, but even if every $d$ of them intersect, there still can be no finite bound for the minimum number of points required
	to hit all of them.

We first give a simple proof for \cref{thm: fractional helly}(1).
To this end, we need the following lemma whose proof is delayed.
\begin{lemma} \label{lemma: covering by distance}
    Let $n > t \ge \delta \ge 0$, $X$ be a finite nonempty set, and $a,b \in X^n$.
    % Write $P:=\{p\in X^n: \dist(a,p)\le \min(\dist(a,b),2t-\delta), \dist(b,p) \le 2t\}$.
    % Then, there is a set of $\binom{4t-\delta}{t-\delta}|X|^{t-\delta}$ points in $X^n$ hitting all the Hamming balls $B(p,t)$, $p \in P$.
    % Then, there is a set of $\binom{4t-\delta}{t-\delta}|X|^{t-\delta}$ points in $X^n$ such that for every $p \in X^n$, where $\dist(a,p)\le \min(\dist(a,b),2t-\delta)$ and $\dist(b,p) \le 2t$, the Hamming ball $B(p,t)$ contains at least one of the points.
    Then, there is a set of $\binom{4t-\delta}{t-\delta}|X|^{t-\delta}$ points in $X^n$ hitting all the Hamming balls $B(p,t)$ with $\dist(a,p)\le \min(\dist(a,b),2t-\delta)$ and $\dist(b,p) \le 2t$.
\end{lemma}
We remark that when $t=\delta$, we do not require $|X|<\infty$ because then, all $B(p,t)$s under consideration contains $a$.
\begin{proof}[Proof for \cref{thm: fractional helly}(1)]
    We may assume $m \ge 12$ as otherwise $\alpha > 1$.
    Construct a graph $G$ with vertex set $V(G)=[m]$ where $i$ and $j$ are adjacent if $B_i$ and $B_j$ intersect, i.e. $\dist(a_i,a_j) \le 2t$.
    Starting from $G$, by iteratively deleting vertices of degree 
	smaller than $\alpha(m-1)/2$ as long as there are such vertices, 
	we arrive at an induced subgraph $G'$ of $G$.
    By assumption, $e(G) \ge \alpha\binom{m}{2}=m\cdot \alpha(m-1)/2$.
    This means $G'$ is not empty and hence, the minimum degree of $G'$ is at least $\alpha(m-1)/2 \ge \alpha m/3$ (using $m \ge 12$).
    Fix any vertex $u \in V(G')$.
    The number of paths $uvw$ in $G'$ is at least $\alpha m/3\cdot(\alpha m/3-1)\ge \alpha^2m^2/12$, using $\alpha \ge 12/m$.

    An (ordered) triple of distinct vertices $(x,y,z) \in V(G')^3$ is said to be {\em good} if $\dist(a_x,a_z) \le \min(\dist(a_x,a_y),2t)$ and $\dist(a_y,a_z) \le 2t$.
    Observe that for any path $uvw$ in $G$,
    \begin{itemize}
        \item if $uw \notin E(G)$, then $\dist(a_v,a_w) > 2t$ and $\dist(a_u,a_v) \le 2t < \dist(a_u,a_w)$, so $(u,w,v)$ is good; 
        \item if $uw \in E(G)$, then $\dist(a_v,a_w) \le 2t, \dist(a_u,a_v)\le\dist(a_u,a_w)\le 2t$ (so $(u,w,v)$ is good) or $\dist(a_v,a_w) \le 2t, \dist(a_u,a_w)\le\dist(a_u,a_v)\le 2t$ (so $(u,v,w)$ is good).
    \end{itemize}
    Enumerating over all paths of length 2, there are at least $\alpha^2m^2/24$ good triples $(u,v,w)$, where $u$ is fixed (as each good triple is counted at most twice).
    By the pigeonhole principle, there exist $u,v\in [m]$ and $W\subseteq [m]$ such that $|W| \ge \alpha^2 m/24$ and $(u,v,w)$ is good for all $w \in W$.
    Then, \cref{lemma: covering by distance} with $\delta=0, a=a_u,b=a_v, p=a_w$  guarantees $\binom{4t}{t}|X|^{t}$ points in $X^n$ hitting every $B_w$, $w \in W$.
    Therefore, some point among these $\binom{4t}{t}|X|^{t}$ points hits at least $\frac{\alpha^2 m}{24\binom{4t}{t}|X|^t}=\Omega(\alpha^2m|X|^{-t}\big/\binom{4t}{t})$ balls.
\end{proof}

We now provide the following definitions that are useful in the proof of \cref{thm: fractional helly}(1) and of \cref{thm: hadwiger-debrunner}.
% The following definitions will be used in the proof
% of the second part of \cref{thm: fractional helly}.
\begin{definition}\label{def: free dim}
    Let $m\ge 0, n>t\ge 0$, let 
	$X$ be a nonempty set, and $a_1,\dots,a_m\in X^n$.
    Define $\free(a_1,a_2,\dots,a_m;t)$ to be the largest size of $K\subseteq [n]$ such that for some $w\in X^n$, $\dist(\restrict{w}{K^c},\restrict{a_i}{K^c})+|K| \le t$ for all $i \in [m]$.
    Define $\free(a_1,a_2,\dots,a_m;t):=-\infty$ if no such $K$ exists.
    % and $\free(;t)=t+1$, i.e. $m=0$.
    % If no such set exists, we define $\free(a_1,\dots,a_m;t):=-\infty$; if $m=0$, we define $\free(;t)=t+1$.
\end{definition}
This definition says that, without looking at the coordinates indexed by $k \in K$, there exists some point lying in all the Hamming balls of radius $t-|K|$ centered at $a_i$, $1 \le i \le m$.
In other words, we can freely choose the coordinates in $K$ for $w$ while maintaining that $w \in \bigcap_{i=1}^m B(a_i,t)$.
We note that $\free(a_1,\dots,a_{m+1};t) \le \free(a_1,\dots,a_{m};t)\le t$, $\free(a_1;t)=t$ and $\free(a_1,\dots,a_m;t) \ge 0$ if and only if $\bigcap_{i=1}^m B(p_i,t) \neq \emptyset$.
In addition, we can assume that $w_k\in\{a_{1,k},a_{2,k},\dots,a_{m,k}\}$ for all $k\in [n]\setminus K$ because otherwise, we should have considered $K':=K\cup \{k\}$.
This motivates the following definition.
\begin{definition} \label{def: candidates given by free dim}
    Let $m\ge 0, n>t\ge 0$, let $X$ be a nonempty set, and $a_1,\dots,a_m\in X^n$.
    Define $W(a_1,a_2,\dots,a_m;t)$ to be the set of $w \in \bigcap_{i=1}^m B(a_i,t)$ where $w_k\in\{a_{1,k},a_{2,k},\dots,a_{m,k}\}$ for all $k\in[n]$.
\end{definition}
% Intuitively, $W(a_1,\dots,a_m)$ is the set of candidate points of $w$ in \cref{def: free dim}.
When it is clear from the context, we omit $t$ in $\free(\cdot)$ and $W(\cdot)$.
The following crucial property, whose proof is delayed, shows how 
to use $\varphi(\cdot)$ in order to find a ``small'' set hitting the Hamming balls.
\begin{lemma} \label{claim: properties of W}
    Let $m\ge 1, n>t\ge 0$, let $X$ be any nonempty set, and $a_1,\dots,a_m\in X^n$.
    \begin{enumerate}
        \item[(1)] $|W(a_1,a_2\dots,a_m)| \le (em)^t$;
        \item[(2)] $W(a_1,a_2\dots,a_m)$ hits $B(a,t)$ for any $a\in X^n$ with $\free(a_1,a_2,\dots,a_{m},a)=\free(a_1,a_2,\dots,a_m)\ge 0$.
        % Any $a\in X^n$ with $\free(a_1,a_2,\dots,a_{m},a)=\free(a_1,a_2,\dots,a_m)\ge 0$ satisfies that $B(a,t)$ contains some point in $W(a_1,a_2\dots,a_m)$;
    \end{enumerate}
\end{lemma}
% \begin{lemma} \label{claim: size of candidates}
%     Let $m\ge 1, n>t\ge 0$, $X$ be any nonempty set, and $a_1,\dots,a_m\in X^n$.
%     Then, $|W(a_1,a_2\dots,a_m)| \le (em)^t$.
% \end{lemma}

Now, we can prove \cref{thm: fractional helly}(2) and \cref{thm: hadwiger-debrunner}.
\begin{proof}[Proof of \cref{thm: fractional helly}(2)]
    We may assume $m \ge 2t$ without loss of generality.
    An unordered tuple $(i_1,i_2,\dots,i_{t+2})\in\binom{[m]}{t+2}$ is said to be {\em good} if $B_{i_1},B_{i_2},\dots,B_{i_{t+2}}$ intersect.
    Find the largest $\ell \in [t+1]$ such that there exists (distinct) $i_1,i_2\dots,i_\ell \in [m]$ with the following properties.
    \begin{itemize}
        \item $0 \le \free(a_{i_1},a_{i_2},\dots,a_{i_j}) <\free(a_{i_1},a_{i_2},\dots,a_{i_{j-1}})$ for $2 \le j \le \ell$;
        \item there are at least $\frac{t+3-\ell}{t+2}\alpha 
		\binom{m-\ell}{t+2-\ell}$ good tuples containing $i_1,\dots,i_\ell$.
    \end{itemize}
    We first note that $\ell$ is well-defined because when $\ell=1$, the 
	pigeonhole principle implies that some $i_1 \in [m]$ lies in at least $\alpha\binom{m}{t+2}\cdot\frac{t+2}{m}=\alpha\binom{m-1}{t+1}$ good tuples.
    Now, let $I$ be the set of $i_{\ell+1} \in [m]\setminus\{i_1,i_2,\dots,i_\ell\}$ such that at least $\frac{t+2-\ell}{t+2}\alpha \binom{m-\ell-1}{t+1-\ell}$ good tuples contain $i_1,i_2,\dots,i_{\ell+1}$.
    Let $Z$ be the count of the number of good tuples containing $i_1,i_2,\dots,i_\ell$ along with one entry other than $i_1,i_2,\dots,i_\ell$. 
    % By counting the number of good tuples containing $i_1,i_2,\dots,i_\ell$ along with one other entry among $i_{\ell+1},i_{\ell+2},\dots,i_{t+2}$, we get
    It holds that 
    \begin{equation} \nonumber
        \begin{aligned}
            \frac{t+3-\ell}{t+2}\alpha \binom{m-\ell}{t+2-\ell} \cdot (t+2-\ell)
            \le Z \le&\, |I|\binom{m-\ell-1}{t+1-\ell}+(m-\ell-|I|)\frac{t+2-\ell}{t+2}\alpha \binom{m-\ell-1}{t+1-\ell}  \\
            \le&\, |I|\binom{m-\ell-1}{t+1-\ell}+\frac{t+2-\ell}{t+2}\alpha \binom{m-\ell}{t+2-\ell}(t+2-\ell),
        \end{aligned}
    \end{equation}
    implying $|I| \ge \frac{\alpha}{t+2}\binom{m-\ell}{t+2-\ell}(t+2-\ell)\big/\binom{m-\ell-1}{t+1-\ell}=\frac{\alpha}{t+2}(m-\ell)$.
    We claim that $\varphi(a_{i_1},a_{i_2},\dots,a_{i_\ell})=\varphi(a_{i_1},a_{i_2},\dots,a_{i_{\ell+1}})$ for all $i_{\ell+1} \in I$.
    Indeed, if $\ell=t+1$, then 
    $$\varphi(a_{i_1},a_{i_2},\dots,a_{i_\ell})\le \varphi(a_{i_1},a_{i_2},\dots,a_{i_{\ell-1}})-1\le\dots \le \varphi(a_{i_1})-(\ell-1)=  \varphi(a_{i_1})-t = 0,$$ so $0\le \varphi(a_{i_1},a_{i_2},\dots,a_{i_{\ell+1}})\le \varphi(a_{i_1},a_{i_2},\dots,a_{i_\ell})=0$. On the other hand, if $\ell < t+1$, then the claim holds due to the maximality of $\ell$ (otherwise $i_1,i_2,\dots,i_{t+1}$ is a longer sequence).
    Now, \cref{claim: properties of W} guarantees a set of at most $(e\ell)^t$ points in $X^n$ hitting every $B_{i_{\ell+1}}$, $i_{\ell+1} \in I$.
    By the pigeonhole principle, some point in $X^n$ hits at
	least $|I|/(e\ell)^t \ge \frac{\alpha(m-\ell)}{(t+2)(e\ell)^t}
	=\Omega(\alpha m/(e(t+1))^{t+1})$ 
	of the Hamming balls,
	(here, we used that $m \ge 2t$).
\end{proof}

\begin{proof}[Proof of \cref{thm: hadwiger-debrunner}]
    Write $A:=\{a_1,a_2,\dots,a_m\}$.
    Recall that given any $x_1,x_2,\dots,x_k\in A$, $\free(x_1,x_2,\dots,x_k)\ge 0$ if and only if $\bigcap_{i=1}^k B(x_i,t)\neq\emptyset$.
    % For any $k \ge 1$, an ordered $k$-tuple $(x_1,x_2,\dots,x_k)\in X^n$ (here, we do not require the points to be distinct) is said to be {\em good} if $B(x_1,t),B(x_2,t),\dots,B(x_k,t)$ intersect (equivalently, $\free(x_1,x_2,\dots,x_k)\ge 0$).
    Find the largest $\ell \ge 1$ such that there exist $x_1,x_2,\dots,x_\ell \in A$ with the following properties.
    \begin{itemize}
        \item[(i)] For every $1\le i_1<i_2<\dots<i_q \le \ell$, it holds that 
        $\bigcap_{j=1}^q B(x_{i_j},t)=\emptyset$;
        % $(x_{i_1},x_{i_2},\dots,x_{i_q})$ is not a good $q$-tuple;
        \item[(ii)] for every $2 \le k < q$ and $1 \le i_1<i_2<\dots<i_k\le\ell$ with 
        $\bigcap_{j=1}^{k-1} B(x_{i_j},t)\neq\emptyset$, it holds that         $\free(x_{i_1},x_{i_2},\dots,x_{i_{k}})<\free(x_{i_1},x_{i_2},\dots,x_{i_{k-1}})$.
    \end{itemize}
    We first note that $\ell$ is well-defined as one can take $\ell=1,x_1=a_1$.
    In addition, by our assumption, out of any $p$ Hammings balls among $B_1,B_2,\dots,B_m$, $q$ of them intersect, so $\ell \le p-1$.
    % $B(a_1,t),\dots,B(a_m,t)$
    % Also, $\ell \le p-1$ because our assumption guarantees that out of any $p$ Hamming balls in $A$, $q$ form a good tuple.
    Fix any $a \in A$.
    The maximality of $\ell$ implies that $\free(x_{i_1},x_{i_2},\dots,x_{i_{q-1}},a)\ge 0$ for some $1 \le i_1<i_2<\dots<i_{q-1}\le \ell$ or $0\le \free(x_{i_1},x_{i_2},\dots,x_{i_{k}},a)=\free(x_{i_1},x_{i_2},\dots,x_{i_{k}})$
    for some $1 \le k<q-1$ and $1 \le i_1<i_2<\dots<i_k\le \ell$.
    As a consequence, one of the following must hold.
    \begin{itemize}
        \item[(1)] $0\le \varphi(x_{i_1},x_{i_2},\dots,x_{i_k},a)=\varphi(x_{i_1},x_{i_2},\dots,x_{i_k})$ for some $1 \le k < q,1 \le i_1<i_2<\dots<i_k\le\ell$;
        \item[(2)] $0 \le \varphi(x_{i_1},x_{i_2},\dots,x_{i_{q-1}},a)<\varphi(x_{i_1},x_{i_2},\dots,x_{i_{q-1}})$ for some $1 \le i_1<i_2<\dots<i_{q-1}\le\ell$.
    \end{itemize}
    
    We deal with $a \in A$ satisfying (1) and (2) separately.
    For every $\emptyset\neq J \subseteq [\ell]$ of size at most $q-1$, let $W_J:=W(x_{i_j}: j \in J)$ (see \cref{def: candidates given by free dim}).
    By \cref{claim: properties of W}, $|W_J| \le (e|J|)^t\le (\big(e(q-1)\big)^t$ and $W_J$ hits $B(a,t)$ whenever $a \in A$ fulfills (1) with $J=\{i_1,i_2,\dots,i_k\}$.
    Taking the union of all these $W_J$s, $W:=\bigcup_{J} W_J$ satisfies $|W| \le \binom{p-1}{<q}\big(e(q-1)\big)^t\le \binom{p}{<q}\big(e(q-1)\big)^t$ and $W$ hits $B(a,t)$ whenever $a\in A$ fulfills (1).
    % for all $a \in A$ satisfying (1), $B(a,t)$ contains some point in $W$.
    
    Now, suppose $J=\{i_1<i_2<\dots<i_{q-1}\}\subseteq [\ell]$.
    % is any such that $\varphi(x_{i_1},x_{i_2},\dots,x_{i_{q-1}})>0$.
    Consider $A_J$, the set of all $a \in A$ satisfying (2) with $i_1,i_2,\dots,i_{q-1}$.
    We will propose a set $Y_J\subseteq X^n$ that hits every $B(a,t)$, $a \in A_J$.
    To this end, we may assume $\varphi(x_{i_1},x_{i_2},\dots,x_{i_{q-1}})\ge0$ as otherwise $A_J=\emptyset$.
    We also need the following estimate.
    \begin{claim}\label{claim: small distance to candidates}
        For every $a \in A$,
        $\dist(a,W_J):=\min_{w \in W_J}\dist(a,w) \le 2t+2-q$.
    \end{claim}
    \begin{proof}
        Since $\varphi(x_{i_1},x_{i_2},\dots,x_{i_{q-1}},a) \ge 0$, there exists $w\in B(a,t)\cap \bigcap_{j=1}^{q-1} B(x_{i_j},t)$, i.e. $\dist(w,a) \le t$ and $\dist(w,x_{i_j})\le t$ for all $j\in[q-1]$.
        Let $K$ be the set of $k\in[n]$ such that $w_k\notin\{x_{i_1,k},x_{i_2,k}\dots,x_{i_{q-1},k}\}$.
        Then, $\dist(\restrict{w}{K^c},\restrict{x_{i_j}}{K^c})+|K|=\dist(w,x_{i_j})\le t$ for all $j \in [q-1]$.
        Using (ii) and \cref{def: free dim}, we acquire $$|K| \le \varphi(x_{i_1},x_{i_2},\dots,x_{i_{q-1}})\le \varphi(x_{i_1},x_{i_2},\dots,x_{i_{q-2}})-1\le\dots \le t+2-q.$$ 
        Now, pick $w' \in X^n$ where $w'_k=w_k$ for $k \in K^c$ and $w'_k=x_{i_1,k}$ for $k \in K$.
        It satisfies that $w'_k\in\{x_{i_1,k},x_{i_2,k},\dots,x_{i_{q-1},k}\}$ for all $k \in [n]$ and $\dist(w',x_{i_j})\le\dist(\restrict{w'}{K^c},\restrict{x_{i_j}}{K^c}) + |K| =\dist(\restrict{w}{K^c},\restrict{x_{i_j}}{K^c}) + |K|=\dist(w,x_{i_j})\le t$ for all $j\in[q-1]$.
        In other words, $w'\in\bigcap_{j=1}^{q-1} B(x_{i_j},t)$ and hence, $w' \in W_J$.
        So, $\dist(a,W_J)\le \dist(a,w') \le \dist(a,w)+\dist(w,w')\le t+|K|\le 2t+2-q$.
    \end{proof}
    Next, we generate $x_{J,1},x_{J,2},\dots,x_{J,k_J} \in A_J$ as follows.
    Pick any $x_{J,1} \in A_J$; having picked $x_{J,1},x_{J,2},\dots,x_{J,k}$ for some $k \ge 1$, if there exists $a \in A_J$ with $\dist(a,x_{J,i})>2t$ for all $i\in[k]$, pick $x_{J,k+1}$ to be such $a$ that maximizes $\dist(a,W_J)$.
    Clearly, among $B(x_{J,1},t),B(x_{J,2},t),\dots,B(x_{J,k_J},t)$, no two balls intersect, so $k_J < p$.
    For every $w \in W_J$ and every $1 \le k \le k_J$, let $Y_{w,k}$ be the set of points given by \cref{lemma: covering by distance} (plugging $a:=w$, $b:=x_{J,k}$ and $\delta=q-2$); so $|Y_{w,k}|\le \binom{4t+2-q}{t+2-q}|X|^{t+2-q}$.
    We claim that $Y_J:=\bigcup_{w \in W_J,k\in[k_J]}Y_{w,k}$ hits every $B(a,t)$, $a \in A_J$.
    To this end, fix an arbitrary $a \in A_J$.
    According to the generation of $x_{J,1},x_{J,2},\dots,x_{J,k_J}$, there exists $k_a \in [k_J]$ such that $\dist(a,x_{J,k_a})\le 2t$.
    We may take the minimum such $k_a$.
    Thus,$\dist(a,x_{J,k})> 2t$ for all $1 \le k < k_a$.
    But then, the procedure (in step $k_a$) also implies $\dist(a,W_J)\le \dist(x_{J,k_a},W_J)$.
    Taking $w \in W_J$ such that $\dist(a,W_J)=\dist(a,w)$, we acquire $\dist(a,w)=\dist(a,W_J)\le\dist(x_{J,k_a},W_J)\le\dist(x_{J,k_a},w)$.
    Recall from \cref{claim: small distance to candidates} that $\dist(a,w)=\dist(a,W_J) \le 2t+2-q$.
    Altogether, $\dist(w,a)\le\min(\dist(w,x_{J,k_a}),2t+2-q)$ and $\dist(a,x_{J,k_a}) \le 2t$.
    By \cref{lemma: covering by distance}, $Y_{w,t}$ hits $B(a,t)$ and hence, $Y_J$ hits $B(a,t)$.
    % $B(a,t)$ contains some points in $Y_{w,k} \subseteq Y_J$.
    In addition, 
    $$
        |Y_J| \le \sum_{w\in W_J,k\in[k_J]} |Y_{w,k}|< |W_J|k_J\binom{4t+2-q}{t+2-q}|X|^{t+2-q}\le (eq)^t p2^{O(t)}|X|^{t+2-q}
        \le q^tp2^{O(t)}|X|^{t+2-q}.
    $$

    To complete the proof, we consider two cases.
   If $q=t+2$, note that all $a \in A$ satisfy (1). 
Indeed, $a \in A$ satisfying (2) is not possible, since then (2) and (ii) 
imply
\begin{eqnarray*}
   0 &\le& \varphi(x_{i_1},x_{i_2},\dots,x_{i_{q-1}},a)\le \varphi(x_{i_1},x_{i_2},\dots,x_{i_{q-1}})-1\le \varphi(x_{i_1},x_{i_2},\dots,x_{i_{q-2}})-2\le \cdots\\
   &\le& 
  \varphi(x_{i_1})-(q-1)=t-(q-1)=-1. 
\end{eqnarray*}
In other words, $a_1,a_2,\dots,a_m$ all satisfy (1).
By the former discussion, $W$ hits every $B_i$, $i\in[m]$, where 
$$
    |W| 
    \le \binom{p}{\le t+1}(e(t+1))^t 
    = O\!\left(\left(\frac{ep}{(t+1)}\right)^{t+1} (e(t+1))^t \right) 
    = O(e^{2t}p^{t+1}).
$$

If $2 \le q \le t+1$. Either $a \in A$ satisfies (1), so $W$ hits $B(a,t)$, or $a \in A$ satisfies (2), so $Y_J$ hits $B(a,t)$ for some $J \subseteq [\ell]$ of size $q-1$.
Thus, $Y:=W \cup \bigcup_{J} Y_J$ is the desired set, whose size 
$$
    |Y| 
    \le \binom{p}{<q}\big(e(q-1)\big)^t+\binom{p}{q-1}q^tp2^{O(t)}|X|^{t+2-q}
    =p^qq^t|X|^{t+2-q}2^{O(t)}.
$$
\end{proof}

\begin{proof}[Proof of \cref{lemma: covering by distance}]
    Write $P:=\{p\in X^n: \dist(a,p)\le \min(\dist(a,b),2t-\delta), \dist(b,{p}) \le 2t\}$.
    We may assume that $P\neq \emptyset$ and that $\dist(a,b) > t$ as otherwise every $B(p,t)$, $p \in P$, contains $a$.
    By taking any $p\in P$, we know $\dist(a,b)\le \dist(a,{p})+\dist({p},b)\le 4t-\delta$.
    % Then, $\dist(a,b)\le \dist(a,\tilde{p})+\dist(\tilde{p},b)\le 4t-\delta$.
    Write $D:=\{k\in[n]:a_k\neq b_k\}$, so $|D|=\dist(a,b)\le 4t-\delta$.
    % Hence, $d=\dist(a,b)\le \dist(a,\tilde{p})+\dist(\tilde{p},b)\le 4t-\delta$.
    Let $Y$ be the set of all $y\in X^n$ such that $\{k: y_k\neq a_k\}$ is a subset of $D$ of size at most $t-\delta$; $|Y| \le \binom{4t-\delta}{t-\delta} |X|^{t-\delta}$.
    We prove that $Y$ has the desired property.
    To this end, fix any $p \in P$, and we will find some $y \in Y$ hitting $B(p,t)$.
    Writing $D_p := \{k\in[n]: a_k \neq p_k\}$, it holds that $|D_p|=\dist(a,p)\le \min(|D|,2t-\delta)\le |D|$, thereby $|D { \triangle}D_p| \ge 2|D_p\setminus D|$.
    Now, observe that $b_k \neq p_k$ for all $k \in D\triangle D_p$.
    We then acquire $2|D_p\setminus D|\le |D\triangle D_p|\le\dist(b,p)\le 2t$, i.e. $|D_p\setminus D|\le t$.
    % $2t \ge \dist(b,p)\ge |D\triangle D_p|\ge 2|D_p\setminus D|$.
    % This is impossible.
    % , whose size is larger than $2t$.
    % In other words, $\dist(b,p)>2t$; this is impossible.
    % Thus, $|D_p\setminus D| \le t$ must hold.
    Now, take any $I\subseteq D_p\cap D$ of size $\min(|D_p\cap D|, t-\delta))$ and $y \in X^n$ such that $y_k=a_k$ for all $k \in [n]\setminus I$ and $y_k=p_k$ for all $k \in I$.
    We know $y \in Y$ (because $|I|\le t-\delta$) and $\dist(y,p)=|D_p\setminus I| \le t$. Indeed, if $|I|=t-\delta$ it holds because $|D_p|\leq 2t-\delta$, otherwise $|I|=|D_p\cap D|$ and hence $|D_p\setminus I|=|D_p\setminus D| \leq t$.
    In other words, $y\in Y$ hits $B(p,t)$, completing the proof.
\end{proof}
\begin{proof}[Proof of \cref{claim: properties of W}]
    Write $W:=W(a_1,a_2,\dots,a_m)$.
    For (1), we may assume $W \neq\emptyset$ and $m \ge 2$ because $W=\{a_1\}$ when $m=1$.
    For each $k\in[n]$, denote $V_k:=\{a_{1,k},a_{2,k},\dots,a_{m,k}\}$.
    Consider $\sum_{i=1}^m \dist(a_i,\tilde{w})$ for an arbitrary $\tilde{w} \in W$, which counts pairs $(i,k) \in [m]\times [n]$ where $a_{i,k}\neq\tilde{w}_k$.
    For each $k \in [n]$, there are at least $|V_k|-1$ indices $i\in[m]$ with $a_{i,k}\neq w_k$, so $ \sum_{k=1}^n (|V_k|-1) \le \sum_{i=1}^m \dist(a_i,\tilde{w})\le mt$.
    Now, observe that any $w \in W$ has that $w_k\neq a_{1,k}$ for at most $t$ of $k\in[n]$.
    Thus, we can enumerate over set of these indices $k$, which we denote by $S\subseteq [n]$, and for each such $k \in S$, there are $|V_k|-1$ choices for $w_k$, i.e. 
    \begin{equation} \nonumber
        \begin{aligned}
            |W| 
            \le \sum_{S \subseteq [n], |S|\le t} \prod_{k\in S} (|V_k|-1)
            \le \sum_{s=0}^t \frac{1}{s!}\!\left(\sum_{k=1}^n (|V_k|-1)\right)^s 
            \le \sum_{s=0}^t \frac{(mt)^s}{s!}
            \le 2\frac{(mt)^t}{t!}
            \le (em)^t.
        \end{aligned}
    \end{equation}
    Here, we used the Stirling's approximation $t!\ge 2(t/e)^t$ for $t \ge 1$.

    For (2), let $K\subseteq [n],w \in X^n$ be the set and vector 
	in the definition of $\free(a_1,a_2,\dots,a_{m},a)$.
    % so $|K|=\free(a_1,a_2,\dots,a_{m},a)$.
    Putting $w_k=a_{1,k}$ for all $k \in K$, we have $\dist(a,w)\le t$ and $\dist(a_i,w) \le t$ for all $i\in [m]$.
    It suffices to show that $w \in W$.
    Suppose not, i.e. $w_k\notin\{a_{1,k},a_{2,k},\dots,a_{m,k}\}$ for some $k\in[n]$.
    % i.e. for some $k \in [n]$, $w_k\neq a_{i,k}$ for all $i\in[m]$.
    Clearly, $k \notin K$.
    Putting $L:=K\cup \{k\}$, it holds that 
    $\dist(\restrict{w}{L^c},\restrict{a_i}{L^c}) + |L|
        = \dist(\restrict{w}{K^c},\restrict{a_i}{K^c}) - 1 + |K| + 1
        \le t$ for all $1 \le i\le m$.
    This means $\free(a_1,a_2,\dots,a_m) \ge |L| > \free(a_1,a_2,\dots,a_{m},a)$, contradicting our assumption.
    Thus, $w \in W\cap B(a,t)$, as desired.
\end{proof}

\subsection{Sequences of sets}
One way to generalize \cref{thm: main result for f} is to consider sequences of sets.
More precisely, given $n > t \ge 0$, $a,b \ge 1$ and a set $X$ with $|X| \ge a+b$, an {\em $(n,t,a,b,X)$-system} is a collection of pairs $({A}_i,{B}_i)_{i\in[m]}$ (for some $m$) such that for each $i$, ${A}_i=(A_{i,1},A_{i_2},\dots,A_{i,n})$ (similarly, ${B}_i=(B_{i,1},B_{i_2},\dots,B_{i,n})$) where each $A_{i,k}$ is a subset of $X$ of size at most $a$ (similarly, each $B_{i,k}$ is a subset of $X$ of size at most $b$).
Define the distance $\dist(A_i,B_j)$ to be the number of $k \in [n]$ such that $A_{i,k}\cap B_{j,k}=\emptyset$. 
Then, we can extend $f(t;X)$ by denoting $f(n,t,a,b;X)$ to be the size of the largest $(n,t,a,b,X)$-system such that $\dist(A_i,{B}_j) \ge t+1$ if and only if $i=j$.
% all $i \in [m]$ and $\dist({A}_i,{B}_j) \le t$ for all distinct $i,j \in [m]$.
One can check that \cref{thm: general theorem for f} corresponds to the case $a=b=1$ by replacing each entry of $a_i$s and $b_i$s by a singleton containing it.
% Now, we can extend $f(n,t,X)$ as follows.
% \begin{definition}
%     Let $n > t \ge 0$, $a,b \ge 1$ and $X$ be a set with $|X| > \max(a,b)$.
%     Define $f(n,t,X,a,b)$ to be the maximum $m$ such that there exists $(n,t,X,a,b)$-system $({A}_i,{B}_i)_{i\in[m]}$ such that $\dist(A_i,{B}_i) \ge t+1$ all $i \in [m]$ and $\dist({A}_i,{B}_j) \le t$ for all distinct $i,j \in [m]$.
% \end{definition}
The first author~\cite{Alon85} proved that $f(t+1,t,a,b;X)= \binom{a+b}{a}^{t+1}$ and \cref{thm: main result for f} implies $f(n,t,1,1;X) = 2^{t+1}$.
One natural guess might be that $f(n,t,a,b;X)=\binom{a+b}{a}^{t+1}$ for all $n > t \ge 0$. 
In particular, this would mean that $f(n,t,a,b;X)$ is independent of $n$.
However, this is not the case whenever $a>1$ or $b>1$.
\begin{proposition} \label{prop: sequence of sets}
    $\binom{n}{t+1} \left(\binom{a+b}{b}-2\right)^{t+1}\!\!\!\!\le f(n,t,a,b;X) \le \binom{n}{t+1} \binom{a+b}{b}^{t+1}$ if $n>t \ge 0$ and $|X| \ge a+b$.
\end{proposition}
\begin{proof}
    For the upper bound, suppose $({A}_i,{B}_i)_{i\in [m]}$ is an 
	$(n,t,a,b,X)$-system realizing $f(n,t,a,b;X)$.
    Uniformly sample a subset $S \subseteq [n]$ of size $t+1$ and consider the following $(t+1,t,a,b,X)$-system: for each $i \in [m]$, ${A}'_i:=(A_{i,k})_{k\in S}$ and ${B}'_i:=(B_{i,k})_{k \in S}$.
    Clearly, $\dist({A}'_i,{B}'_j) \le \dist({A}_i,{B}_j) \le t$ for every distinct $i,j\in I$.
    Let $I$ be the set of $i \in [m]$ where $\dist({A}'_i,{B}'_i) \ge t+1$.
    We know $|I| \le f(t+1,t,a,b;X)=\binom{a+b}{a}^{t+1}$.
    Also, using that $\Pr[i \in I] \ge 1/\binom{n}{t+1}$, we conclude that
	${m}/{\binom{n}{t+1}} \le \EE|I| \le \binom{a+b}{a}^{t+1}$, i.e. $f(n,t,a,b;X)=m \le \binom{n}{t+1} \binom{a+b}{b}^{t+1}$.

    For the lower bound, we may assume $X=[a+b]$.
    Let $S_1,\dots,S_\ell$ be an arbitrary enumeration of all subsets of $X$ of size $a$ (so $\ell=\binom{a+b}{a}$), and for each $i\in[\ell]$, let $T_i:= X\setminus S_i$.
    Observe that $S_i\cap T_j=\emptyset$ if and only if $i=j$.
    Define a mapping $\varphi:[\ell-1]\to [\ell]$ by putting $\varphi(\ell-1)=\ell$ and $\varphi(i)=i$ for all $i\in[\ell-2]$.
    Now, let $a_1,\dots,a_m$ be any enumeration of the sequences in $[\ell-1]^n$, exactly $n-t-1$ entries of which are $\ell-1$.
    Then, $m=\binom{n}{t+1}(\binom{a+b}{a}-2)^{t+1}$.
    For each $i\in[m]$, define ${A}_i=(A_{i,k})_{k\in[n]}$ where $A_{i,k}=S_{a_{i,k}}$ and ${B}_i=(B_{i,k})_{k\in[n]}$ where $B_{i,k}=T_{\varphi(a_{i,k})}$.
    Since $({A}_i,{B}_i)_{i\in[m]}$ is a $(n,t,a,b,X)$-system, it suffices to check that $\dist({A}_i,{B}_j) \ge t+1$ if and only if $i=j$.
    For any $i,j\in[m]$ and $k\in[n]$, it holds that $A_{i,k}\cap B_{j,k}=\emptyset \Leftrightarrow S_{a_{i,k}} \cap T_{\varphi(a_{j,k})} = \emptyset \Leftrightarrow a_{i,k}=a_{j,k} \in [\ell-2]$.
    Thus, 
    $$
        \dist({A}_i,{B}_j)
        = \abs{\{k: A_{i,k}\cap B_{j,k}=\emptyset \}}
        = \abs{\{k: a_{i,k}=a_{j,k} \in [\ell-2]\}}
        \le \abs{\{k: a_{i,k}\in [\ell-2]\}} = t+1.
    $$
    Clearly, if $i=j$, then $\dist(\mc{A}_i,\mc{B}_i) = t+1$.
    On the other hand, if $\dist(\mc{A}_i,\mc{B}_j) \ge t+1$, it must be that $a_{j,k}=a_{i,k}$ for all $k$ with $a_{i,k}\in[\ell-2]$.
    This shows $i=j$.
\end{proof}
Notably, when $b = 1$ and $|X|=a+1$, $f(n,t,a,1;X)$ is equal to the maximum $m$ such that there exist $a_1,a_2,\dots,a_m,b_1,b_2,\dots,b_m\in X^n$ where $\dist(a_i,b_i)\le n-t-1$ for all $i$ and $\dist(a_i,b_j) \ge n-t$ for all $i\neq j$.
Indeed, given $A_{i,k}$ with $|A_{i,k}|=a$ and $B_{i,k}$ with $|B_{i,k}|=1$, for each $1 \le i \le m$, we can define $a_i=(a_{i,k})_k$ where $a_{i,k}$ is the only element in $X\setminus A_{i,k}$, and $b_i=(b_{i,k})_k$ where $b_{i,k}$ is the only element in $B_{i,k}$.
Then, $\dist({A}_i,B_j)=n-\dist(a_i,b_j)$.
When $a=1$, the answer is the same as $f(n,t;\{0,1\})=2^{t+1}$ by flipping all the $b_i$s.
But when $a > 1$, \cref{prop: sequence of sets} shows that the largest such family has size $\binom{n}{t+1}(a-O(1))^{t+1}$.
% This is different from \cref{thm: main result for f} (where $\dist(a_i,b_i) \ge t+1$ and $\dist(a_i,b_j) \le t$) when $a>1$ or $b>1$.

\subsection{Connection to the Prague dimension} \label{sec: prague dimension}
Given a graph $G$, the {\em Prague dimension}, $\pd(G)$, is the minimum $d$ such that one can assign each vertex a unique vector in $\mb{Z}^d$ and two vertices are adjacent in $G$ if and only if the two corresponding vectors differ in all coordinates.
In other words, $\pd(G)$ is the minimum $d$ such that there exists some injection $f:V(G) \to \mb{Z}^d$ so that $u, v$ are adjacent in $G$ if and only if $\dist(f(u),f(v))=d$.

The definition and results of the function
$f(n,t,X)$ suggest the following
variant of the Prague dimension.
Given a graph $G$, the {\em threshold Prague dimension}, $\tpd(G)$, is the minimum $t$ such that there exists some $d \in \mb{N}$ and some $f:V(G) \to \mb{Z}^d$ so that $u,v$ are adjacent in $G$ if and only if $f(u)$ and $f(v)$ differ in at least $t$ coordinates, that is, 
$\dist(f(u),f(v)) \ge t$.
By definition, $\tpd(G) \le \pd(G)$.
In this section, we list and compare some properties of 
these two dimensions.

First, $\tpd(G(n,1/2))=\Theta(n/\log n)$ with high probability.
The upper bound holds as $\pd(G(n,1/2))=\Theta(n/\log n)$ with high probability by~\cite{GPW23}.
For the lower bound, let $\mc{G}$ be the set of all graphs with vertex set $[n]$ whose complement has diameter 2.
It is well known that $|\mc{G}|=(1-o(1))2^{\binom{n}{2}}$ (see for example \cite{Bollobas01}).
We will compare the number of ``essentially distinct'' mappings $f:[n]\to\mb{Z}^d$ (that realize $\tpd(G)$ for some $G\in\mc{G}$) to $|\mc{G}|$.
Given any graph $G \in \mc{G}$ and any $f:V(G)\to \mb{Z}^d$ that realizes $t:=\tpd(G)$, without loss of generality, we can assume that 
$f(u) \in [n]^d$ for every vertex $u$ and that 
$f(1)$ is the all-ones string.
Knowing that $\dist(f(u),f(v)) < t$ for $u,v$ not adjacent in 
$G$ and that the diameter of the complement of $G$ is 2, it follows straightforwardly that $\dist(f(1),f(u)) < 2t$ for all $u \in V(G)$.
Define $I_u:=\{k\in[d]: f(1)_k\neq f(u)_k\}$ for each $u \in [n]$ and $I:=\bigcup_{u\in [n]} I_u$.
We have $|I_u| < 2t$ and hence, $|I|<2tn$.
% For each $u\in[n]$, define $I_u:=\{k\in[d]: f(1)_k\neq f(u)_k\}$. We have $|I_u| < 2t$.
% Let $I:=\bigcup_{u\in [n]} I_u$; $|I|<2tn$.
Then, $f(u)_k=f(1)_k=1$ for every $u \in [n]$ and $k \in [d]\setminus I$.
Thus, we can assume $d = 2tn$ without loss of generality.
By the former discussion, we can specify any graph $G \in \mc{G}$ by a spanning tree $T$ in its complement and the corresponding lists of 
length $d=2tn$ so that $f(1)$ is all-ones, $f(u)\in[n]^{2tn}$ for all $u\in[n]$ and $\dist(f(u),f(v))<t$ for all $u,v \in E(T)$.
If $u$ is the parent of $v$ in $T$ and $f(u)$ has been fixed, there are at most $\binom{2tn}{<t}n^t$ choices for $f(v)$.
Thus, the number of such functions $f$ is at most 
$n^{n-2} (\binom{2tn}{<t} n^t)^{n-1}= n^{O(tn)}$ (here $n^{n-2}$ is the number of spanning trees in $K_n$).
If $\tpd(G(n,1/2)) \le t$ with high probability, then $n^{O(tn)}>|\mc{G}|-o(2^{\binom{n}{2}})=(1-o(1))2^{\binom{n}{2}}$, i.e. $t = \Omega(n/\log n)$.

Second, if $u_1,u_2,\dots,u_s$ and $v_1,v_2,\dots,v_s$ are two 
sequences of vertices in $G$ such that $u_i,v_j$ are adjacent 
in $G$ if and only if $i=j$, i.e., the edges $(u_i,v_i)$ form an 
induced matching in $G$, then $\tpd(G) \ge \log_2 s$ by \cref{thm: main result for f}. 
Indeed, any $f:V(G)\to \mb{Z}^d$ realizing $\tpd(G)$ satisfies $\dist(f(u_i),f(v_i)) \ge \tpd(G)$ for all $i\in[s]$ and $\dist(f(u_i),f(v_j))<\tpd(G)$ for all distinct $i,j\in[s]$.
This argument has been widely used to give lower bounds for $\pd(G)$ for various graphs $G$.
For example, let us consider graphs on $n$ vertices such that the minimum degree is at least one while the maximum degree is $\Delta$.
This includes a lot of basic graphs like perfect matchings, cycles, paths, etc.
The first author~\cite{Alon86} showed that the Prague dimension for these graphs is at least $\log_2 \frac{n}{\Delta}-2$ because they contain an induced matching of size at least $\frac{n}{4\Delta}$.
Now, \cref{thm: main result for f} shows the same bound also holds for the threshold Prague dimension.
To compare, we note that Eaton and R\"odl~\cite{ER96} showed that the Prague dimension (and thus the threshold Prague dimension) for these graphs is at most $O(\Delta\log_2 n)$.

Third, the threshold Prague dimension can be much smaller than the Prague dimension.
For example, it is known that $\pd(K_n+K_1)=n$ (see~\cite{LNP80}), 
where $K_n+K_1$ is the vertex disjoint union of 
a clique of size $n$ and an isolated vertex.
However, by mapping the vertices of $K_n$ to the standard orthonormal basis of $\mc{R}^n$ and that of $K_1$ to the all-zeros vector, 
we observe that $\tpd(K_n+K_1)\le 2$.
A more interesting example is the Kneser graph: for $n \ge k$, the Kneser graph $K(n,k)$ is the graph whose vertices are all the $k$-element subsets
of $[n]$ and whose edges are pairs of disjoint subsets.
When $1 \le k \le n/2$, it is known that $\log_2\log_2\frac{n}{k-1} \le \pd(K(n,k)) \le C_k\log_2\log_2 n$ for some constant $C_k$; see \cite{Furedi00}.
For the threshold Prague dimension, define $f:\binom{[n]}{k}\to\{0,1\}^n$ by mapping each vertex in $K(n,k)$ to the indicator vector of length $n$ of the corresponding subset.
Then, for two adjacent vertices $u,v$, (where the corresponding 
two subsets are disjoint), $\dist(f(u),f(v))=2k$. For two 
non-adjacent vertices $u,v$, the two subsets intersect and $\dist(f(u),f(v)) \le 2(k-1)<2k-1$.
This shows $\tpd(K(n,k)) \le 2k-1$.
In addition, $K(2k,k)$ is an induced matching of size $\frac{1}{2}\binom{2k}{k}$, so $\tpd(K(2k,k))\ge \log_2\frac{1}{2}\binom{2k}{k}=2k-O(\log_2 k)$.
Knowing that $K(n,k)$ contains $K(2k,k)$ as an induced subgraph, we have 
$\tpd(K(n,k)) \ge \tpd(K(2k,k))=2k-O(\log_2 k)$.
Thus, $\tpd(K(n,k))$ is asymptotically $2k$. This holds independently of $n$, very different from the behavior of $\pd(K(n,k))$.

Finally, it would be also interesting to determine the maximum possible threshold Prague dimension for an $n$-vertex graph $G$.  For the ordinary Prague dimension, this was done by Lov\'{a}sz, Ne\v{s}et\v{s}il and Pultr~\cite{LNP80}, who showed that $\pd(G)\le n-1$ and $\pd(G)=n-1$ if and only if $G=K_{n-1}+K_1$ (when $n\ge 5$). As we already mentioned above, $K_{n-1}+K_1$ is not a good candidate for maximizing $\tpd(G)$ since $\tpd(K_{n-1}+K_1)\le 2$.
Another natural graph to consider is $K_m+K_m$ when $n=2m$. We claim that $\tpd(K_m+K_m)=\pd(K_m+K_m)=m$.
Let the vertex sets of the two cliques be 
$U=\{u_1,\dots,u_m\}$ and $V=\{v_1,\dots,v_m\}$.
For the upper bound, assign to $u_i$ an all-$i$s string of length $m$, and to $v_i$ a string $s$ of length $m$ starting from $i$ in which $s_k=s_{k-1}+1 \mod m$ for all $k$. 
For the lower bound of $\tpd(K_m+K_m)$, suppose $f:U\cup V \to \mb{Z}^d$ (for some $d$) realizes $t:=\tpd(K_m+K_m)$.
Let 
$$
    C_1 := \sum_{1\le i<j\le m}\dist(f(u_i),f(u_j)) + \sum_{1\le i<j\le m}\dist(f(v_i),f(v_j)),
    \quad C_2 := \sum_{i,j=1}^m \dist(f(u_i),f(v_j)).
$$
We consider $C_1-C_2$.
Fix $k \in [d]$. 
For each $a \in \mb{Z}$, let $s_a$ be the number of $i\in[m]$ such that $f(u_i)_k=a$, and $t_a$ be the number of $i\in[m]$ such that $f(v_i)_k=a$.
The contribution to $C_1-C_2$ from the $k$th coordinates is given by 
$$
    \left(\binom{m}{2}-\sum_a \binom{s_a}{2}\right) + 
    \left(\binom{m}{2}-\sum_a \binom{t_a}{2}\right) -
    \left(m^2 - \sum_a s_at_a \right)
    = \sum_a s_at_a-\frac{s_a^2+t_a^2}{2}
    \le 0.
$$
Summing over all $k\in [d]$, we know $C_1 \le C_2$.
But then, $2\binom{m}{2}\cdot t\le C_1 \le C_2 \le m^2\cdot (t-1)$, showing $\tpd(K_m+K_m)=t \ge m$, as claimed.
It might be the case that $\tpd(G) \le \ceil{\frac{n}{2}}$ for all 
$n$-vertex graphs $G$.

\section{Concluding remarks and open problems} \label{sec: concluding}
In \cref{thm: general theorem for f} we showed there are at most $2^{t+1}$ pairs of $(a_i,b_i)$ such that $\dist(a_i,b_i)\ge t+1$ for all $i$ and $\dist(a_i,b_j) \le t$ for all $i\neq j$.
Consider any $t \ge 1$, the nontrivial case.
Notice that $\frac{V_{t+s,s}}{2^{t+s}}\ge \frac{V_{t+3,3}}{2^{t+3}}=\frac{t+4}{2^{t+3}}>2^{-t-1}$ for all $s \ge 3$.
Therefore, \cref{eq: main eq for general X} indicates that in the extremal case when there are $2^{t+1}$ such pairs, it must be that 
$\dist(a_i,b_i)\in\{t+1,t+2\}$ for all $i$. By taking $a_i\in\{0,1\}^{t+1}$ and $b_i=\overline{a_i}$, we construct $2^{t+1}$ such pairs with $\dist(a_i,b_i)=t+1$. Also by taking $a_i\in\{0,1\}^{t+2}$ with an even number of 1s and $b_i=\overline{a_i}$, we construct $2^{t+2}$ pairs with $\dist(a_i,b_i)=t+2$.
Thus, $\dist(a_i,b_i)=t+1$ and $\dist(a_i,b_i)=t+2$ are both possible in the extremal case. 
It would be interesting to have a further characterization of the extremal cases.

In the realm of the set-pair inequalities, the skew version also plays an important role; see \cite{Lovasz77,Lovasz79}.
Given $t$ and $X$, what is the largest $m$ such that there exist $n\ge t+1$ and $m$ pairs $a_i,b_i \in X^n, 1 \leq i\leq m$ so that 
$\dist(a_i,b_i) \ge t+1$ for all $i \in [m]$ and $\dist(a_i,a_j) \le t$ for all $1 \le i < j \le m$?
We suspect the answer is also $2^{t+1}$, and it would be interesting 
to try to adapt the dimension argument in order to prove it.

In \cite{Furedi84}, F\"uredi showed the set-pair inequality via the the following vector space generalization.
If $A_1,A_2,\dots,A_m$ are $a$-dimensional and $B_1,B_2,\dots,B_m$ are $b$-dimensional linear subspaces of $\mb{R}^n$ such that $\dim (A_i\cap B_j)\le k$ if and only if $i=j$, then $m \le \binom{a+b-2k}{a-k}$.
We wonder if there is a natural generalization of \cref{thm: main result for f} or even \cref{thm: general theorem for f} to vector spaces.

It will be interesting to study the threshold Prague dimension of graphs further. 
In particular, it will be nice to determine or estimate the maximum possible value of this invariant for a graph with $n$ vertices and maximum degree $\Delta$. 
In the case of Prague dimension, Eaton and R\"odl~\cite{ER96} showed the maximum possible dimension of a graph with $n$ vertices and maximum degree $\Delta$ is at most $O(\Delta\log n)$ and at least $\Omega(\frac{\Delta\log n}{\log \Delta + \log\log n})$.

\vspace{0.2cm}

\noindent
{\bf Acknowledgment} We thank Matija Buci\'c and Varun Sivashankar for helpful discussions. 
We also thank Istv\'an Tomon for his useful comments and in particular 
for suggesting to discuss the connection to the study of Radon numbers of  
abstract convexity spaces.

\end{document}